\documentclass[12pt,reqno]{amsart}

\usepackage[margin=1.2in]{geometry}
\usepackage{graphicx,latexsym,graphics}
\usepackage{amsfonts, amssymb, amsmath, amsthm}
\usepackage{algorithmic, algorithm}

\theoremstyle{plain}
\newtheorem{theorem}{Theorem}[section]
\newtheorem{proposition}[theorem]{Proposition}
\newtheorem{lemma}[theorem]{Lemma}

\theoremstyle{remark}
\newtheorem{remark}[theorem]{Remark}

\theoremstyle{definition}
\newtheorem{definition}[theorem]{Definition}

\numberwithin{equation}{section}

\allowdisplaybreaks

\newcommand{\NN}{\mathbb N}

\newcommand{\CC}{\mathbb C}
\newcommand{\RR}{\mathbb R}
\newcommand{\ZZ}{\mathbb Z}

\newcommand{\DD}{\mathcal D}
\newcommand{\SSS}{\mathcal S}

\newcommand{\supp}{\operatorname{supp}}

\begin{document}

\title[Wiener Amalgam Spaces of Quasianalytic Ultradistributions]
{Wiener Amalgam Spaces of Quasianalytic Ultradistributions}
\thanks{The research was partially supported by the bilateral project ``Microlocal analysis and applications'' funded by the Macedonian and Serbian academies of sciences and arts.}

\author[P. Dimovski]{Pavel Dimovski}
\address{P. Dimovski, Faculty of Technology and Metallurgy, Ss. Cyril and Methodius University in Skopje, Ruger Boskovic 16, 1000 Skopje, Macedonia}
\email{dimovski.pavel@gmail.com}

\author[B. Prangoski]{Bojan Prangoski}
\address{B. Prangoski, Faculty of Mechanical Engineering, Ss. Cyril and Methodius University in Skopje, Karpos II bb, 1000 Skopje, Macedonia}
\email{bprangoski@yahoo.com}


\maketitle

\begin{abstract}
We define Wiener amalgam spaces of (quasi)analytic ultradistributions whose local components belong to a general class of translation and modulation invariant Banach spaces of ultradistributions and their global components are either weighted $L^p$ or weighted $\mathcal{C}_0$ spaces. We provide a discrete characterisation via so called uniformly concentrated partitions of unity. Finally, we study the complex interpolation method and we identify the strong duals for most of these Wiener amalgam spaces.
\end{abstract}
\maketitle

\section{Introduction}

The Wiener amalgam spaces (from now, always abbreviated as amalgam spaces) are Banach spaces of tempered distributions, or non-quasianalytic ultradistributions, which measure the global behaviour (usually $L^p$ summability) of the local properties of its elements. In their full generality, they were introduced by Feichtinger in \cite{fe81,feich} (cf. \cite{fei90,FG,fou-ste,he03,felvo}); rudiments of this construction go back to Wiener \cite{wiener,wiener1}. The standard notation is $W(E,C)$, where the local component $E$ is a Banach space of tempered distributions, or, more generally, a Banach space of non-quasianalytic Gelfand-Shilov ultradistributions and the global component $C$ is a Banach space of $L^1_{\operatorname{loc}}$-functions that satisfies additional technical assumptions ($C$ is usually a weighted $L^p$ space). The amalgam space $W(E,C)$ consists of all tempered distributions (or, non-quasianalytic Gelfand-Shilov ultradistributions) $f$ for which the function $x\mapsto\|fT_x\chi\|_E$ belongs to $C$; here, the window $\chi$ is an arbitrary but fixed smooth (or, ultradifferentiable) function with compact support and $T_x$ stands for the operator of translation by $x\in\RR^n$. Nowadays, the amalgam spaces are widely used tool in time-frequency analysis and in the theory of pseudo-differential operators and PDEs; see \cite{acttt,ben-oko,CNT,CR,drw,fei3322,Grochenig,gro-rz,toft1,toft2} and the references therein. We point out that they are closely related to the classical modulation spaces \cite{Grochenig}: denoting by $\mathcal{F}$ the Fourier transform, $W(\mathcal{F}L^p, L^q)$ is the Fourier image of the modulation space $M^{p,q}$  (cf. \cite[Lemma 3.5]{f-p-p}).\\
\indent The goal of this article is to define amalgam spaces of (quasi)analytic ultradistributions, i.e. they will be subspaces of the generalised Gelfand-Shilov spaces of Beurling and Roumieu type $\SSS'^{(M_p)}_{(A_p)}(\RR^n)$ and $\SSS'^{\{M_p\}}_{\{A_p\}}(\RR^n)$, where $\{M_p\}_{p\in\NN}$ and $\{A_p\}_{p\in\NN}$ are two Gevrey type sequences which control the regularity and the decay of the elements of the corresponding test spaces. We require both of them to not grow slower than $\{p!\}_{p\in\NN}$. The local component $E$ will belong to a broad class of Banach spaces of Gelfand-Shilov ultradistributions invariant under translation and modulation \cite{DPPV} (cf. \cite{DPV,lenny}), while the global component will be either a weighted $L^p$ or a weighted $\mathcal{C}_0$ space. We also prove a discrete characterisation of these amalgam spaces via so called \textit{uniformly concentrated partitions of unity} in similar spirit as in \cite{feich}. As we pointed out above, when $\{M_p\}_{p\in\NN}$ is non-quasianalytic, the theory has already been built in \cite{feich}. However, when $\{M_p\}_{p\in\NN}$ is quasianalytic, the classical ideas and techniques are not applicable since they heavily rely on functions with compact support but the test spaces and the local components $E$ may consist entirely of (quasi)analytic functions (even more so, when $M_p=p!$, $p\in\NN$, the Beurling test space $\SSS^{(p!)}_{(A_p)}(\RR^n)$ consists of only entire functions). Because of this, we developed new techniques to overcome this problem. Although we are primarily interested in the (quasi)analytic setting, we do not impose this condition on either one of the sequences $\{M_p\}_{p\in\NN}$ and $\{A_p\}_{p\in\NN}$ so our techniques and results work in the non-quasianalytic setting as well.\\
\indent In the last section, we study the duality and the complex interpolation method of the amalgam spaces.

\section{Preliminaries}\label{notation}

Let $\{M_p\}_{p\in\NN}$ and $\{A_p\}_{p\in\NN}$ be two sequences of positive numbers such that $M_0=M_1=A_0=A_1=1$. Throughout the article, we assume that both sequences satisfy the following conditions:\\
\indent $(M.1)$ $M_{p}^{2} \leq M_{p-1} M_{p+1}$, $p \in\ZZ_+$;\\
\indent $(M.2)$ $M_{p+q} \leq c_0H^{p+q}M_p M_q$, $p,q\in \NN$, for some $c_0,H\geq1$;\\
\indent $(M.6)$ $p!\subset M_p$; i.e., there exist $c_0,L_0\geq1$
such that $p!\leq c_0 L_0^p M_p$, $p\in\NN$.\\
In addition, we assume that $\{A_p\}_{p\in\NN}$ also satisfies the condition\\
\indent $(M.2)^*$ $2m_p\leq m_{pN}$, $p\geq p_0$, for some $p_0,N\in\ZZ_+$, where $m_j:=A_j/A_{j-1}$, $j\in\ZZ_+$.\\
Without loss of generality, we can assume that the constants $c_0$ and $H$ for $\{M_p\}_{p\in\NN}$ and $\{A_p\}_{p\in\NN}$ are the same. It can be shown (see \cite[Proposition 1.1]{Petzsche88}) that $(M.2)^*$ is equivalent to the condition $(M.5)$ employed in \cite{ppv}: there is $q>0$ such that $\{A^q_p\}_{p\in\NN}$ is strongly non-quasianalytic (i.e., $\{A^q_p\}_{p\in\NN}$ satisfies the condition $(M.3)$ of Komatsu \cite{Komatsu1}). The Gevrey sequences $\{p!^{\sigma}\}_{p\in\NN}$, $\sigma\geq 1$, satisfy all of the above conditions. For $\alpha\in\NN^n$, we write $M_{\alpha}=M_{|\alpha|}$. We denote by $M(\cdot)$ and $A(\cdot)$ the associated functions to $\{M_p\}_{p\in\NN}$ and $\{A_p\}_{p\in\NN}$, namely \cite{Komatsu1}: $M(\rho)=\sup_{p\in\NN}\ln(\rho^p/M_p)$, $\rho\geq 0$, and $A(\rho)$ is defined likewise. They are non-negative, continuous, non-decreasing functions that vanish for sufficiently small $\rho$ and grow more rapidly than $\ln \rho^k$ for any $k\in\ZZ_+$ as $\rho\to\infty$ (see \cite{Komatsu1}). When $M_p=p!^{\sigma}$, $\sigma\geq1$, $M(\rho)\asymp \rho^{1/\sigma}$. Throughout the rest of the article, we will frequently tacitly apply the following inequalities
$$
e^{M(\rho+\lambda)}\leq 2e^{M(2\rho)}e^{M(2\lambda)},\,\, \rho,\lambda\geq 0,\quad \mbox{and}\quad e^{2M(\rho)}\leq c_0e^{M(H\rho)},\,\, \rho\geq 0,
$$
(see \cite[Proposition 3.6]{Komatsu1} for a proof of the second) and the analogous ones for $A(\cdot)$.\\
\indent Given $h>0$, we denote by $\SSS^{M_p,h}_{A_p,h}(\RR^n)$ the Banach space of all $\varphi \in \mathcal{C}^{\infty}(\RR^n)$ such that $\sup_{\alpha\in\NN^n} h^{|\alpha|}\|e^{A(h|\cdot|)} \partial^{\alpha}\varphi\|_{L^{\infty}(\RR^n)}/M_{\alpha}<\infty$, and define the locally convex spaces (from now on, abbreviated as l.c.s.)
$$
\SSS^{(M_p)}_{(A_p)}(\RR^n)=\lim_{\substack{\longleftarrow\\ h\rightarrow \infty}} \SSS^{M_p,h}_{A_p,h}(\RR^n) \quad\mbox{and}\quad
\SSS^{\{M_p\}}_{\{A_p\}}(\RR^n)=\lim_{\substack{\longrightarrow\\ h\rightarrow 0^{+}}} \SSS^{M_p,h}_{A_p,h}(\RR^n).
$$
The space $\SSS^{(M_p)}_{(A_p)}(\RR^n)$ is a nuclear Fr\'echet space, while $\SSS^{\{M_p\}}_{\{A_p\}}(\RR^n)$ is a nuclear $(DFS)$-space; furthermore, they satisfy the natural analogue to the Schwartz kernel theorem (see \cite{and-jasson-lenny,ppv}). All of these facts hold equally well even if $\{A_p\}_{p\in\NN}$ does not satisfy $(M.2)^*$. Their strong duals are the spaces of tempered ultradistributions of Beurling and Roumieu type respectively. We employ $\SSS^*_{\dagger}(\RR^n)$ and $\SSS'^*_{\dagger}(\RR^n)$ as a common notation for these spaces.  Throughout the rest of the article, if a separate treatment is needed, we first state assertions for the Beurling cases of $*$ and $\dagger$, followed by the Roumieu cases in parenthesis.\\
\indent We fix the constants in the Fourier transform as $\mathcal{F}f(\xi)=\int_{\RR^n} e^{-2\pi ix \xi}f(x)dx$, $f\in L^1(\RR^n)$. The Fourier transform is a topological isomorphism from $\SSS^*_{\dagger}(\RR^n)$ and $\SSS'^*_{\dagger}(\RR^n)$ onto $\SSS^{\dagger}_*(\RR^n)$ and $\SSS'^{\dagger}_*(\RR^n)$ respectively. Given a Banach space $X\subseteq \SSS'^*_{\dagger}(\RR^n)$, we define its associated Fourier space as the Banach space $\mathcal{F}X\subseteq \SSS'^{\dagger}_{*}(\RR^n)$ with norm $\|\mathcal{F}f\|_{\mathcal{F}X}=\|f\|_X$. We denote by $T_x$, $x\in\RR^n$, and $M_{\xi}$, $\xi\in\RR^n$, the operators of translation and modulation: $T_xf=f(\cdot-x)$, $M_{\xi}f=e^{2\pi i \xi\,\cdot}\,f$. Furthermore, we write $\check{f}(x)=f(-x)$ for reflection about the origin.\\
\indent The proofs of our main results heavily rely on the following multiplicative factorisation of $\SSS^*_{\dagger}(\RR^n)$ \cite{abj} (this is the reason why we impose the condition $(M.2)^*$ on $\{A_p\}_{p\in\NN}$).

\begin{proposition}[{\cite[Corollary 7.10]{abj}}]\label{facto-s-mul}
Let $B$ be a bounded subset of $\SSS^*_{\dagger}(\RR^n)$. There exist $\varphi\in\SSS^*_{\dagger}(\RR^n)$ and a bounded subset $B_1$ of $\SSS^*_{\dagger}(\RR^n)$ such that $B=\varphi B_1$.
\end{proposition}

\indent Given two l.c.s. $X$ and $Y$, we denote by $\mathcal{L}(X,Y)$ the space of continuous linear operators from $X$ into $Y$; $\mathcal{L}_b(X,Y)$ stands for this space equipped with the topology of uniform convergence on all bounded subsets of $X$. When $X=Y$, we abbreviate these notations as $\mathcal{L}(X)$ and $\mathcal{L}_b(X)$ respectively. The notation $X\hookrightarrow Y$ means that $X$ is continuously and densely included into $Y$. Unless otherwise stated, the dual of any l.c.s. $X$ will always carry the strong dual topology; we will sometimes write $X'_b$ when we want to emphasise this.\\
\indent A positive measurable function $\eta:\mathbb{R}^{n}\to (0,\infty)$ is said to be an \textit{ultrapolynomially bounded weight of class $\dagger$}, or simply, a \textit{weight of class $\dagger$}, if there are $C,\tau>0$ (resp., for every $\tau>0$ there is $C>0$) such that $\eta(x+y)\leq C\eta(x)e^{A(\tau|y|)}$, $x,y\in\RR^n$. Given such weight $\eta$, there always exists a continuous weight $\tilde{\eta}$ of class $\dagger$ which is equivalent to $\eta$, i.e. there is $\tilde{C}\geq 1$ such that $\tilde{C}^{-1}\tilde{\eta}(x)\leq \eta(x)\leq \tilde{C}\tilde{\eta}(x)$, $x\in\RR^n$. The definition of an ultrapolynomially bounded weight of class $*$ is analogous (just replace $A(\cdot)$ with $M(\cdot)$). Given a weight $\eta$ of class $\dagger$, we denote by $L^p_{\eta}(\RR^n)$, $1\leq p\leq \infty$, the weighted $L^p$ space of measurable functions $f$ such that $\|f\|_{L^{p}_{\eta}(\RR^n)}:=\|\eta f\|_{L^p(\RR^n)}<\infty$. We also consider the closed subspace $L^{\infty}_{\eta,0}(\RR^n)$ of $L^{\infty}_{\eta}(\RR^n)$ consisting of all $f\in L^{\infty}_{\eta}(\RR^n)$ which additionally satisfy the following property: for every $\varepsilon>0$ there exists a compact set $K\subseteq \RR^n$ such that $\|f\eta\|_{L^{\infty}(\RR^n\backslash K)}\leq \varepsilon$. We denote by $\mathcal{C}_{\eta,0}(\RR^n)$ the space $L^{\infty}_{\eta,0}(\RR^n)\cap \mathcal{C}(\RR^n)$; it is a closed subspace of $L^{\infty}_{\eta,0}(\RR^n)$ and we equip it with the induced norm from $\|\cdot\|_{L^{\infty}_{\eta}}$. When $\eta=1$, we employ the notations $L^{\infty}_0(\RR^n)$ and $\mathcal{C}_0(\RR^n)$ instead of $L^{\infty}_{\eta,0}(\RR^n)$ and $\mathcal{C}_{\eta,0}(\RR^n)$.\\
\indent For a measurable set $K\subseteq \RR^n$, we denote by $\mathbf{1}_K$ the characteristic function of $K$. Similarly, given a countable index set $\Lambda$ and $\tilde{\Lambda}\subseteq \Lambda$, we denote by $\mathbf{1}_{\tilde{\Lambda}}\in\CC^{\Lambda}$ the sequence which is equal to $1$ on $\tilde{\Lambda}$ and $0$ otherwise.\\
\indent Let $h>0$. We denote by $\DD^{M_p,h}_{L^p}(\RR^n)$, $1\leq p\leq \infty$, the Banach space of all $\varphi\in\mathcal{C}^{\infty}(\RR^n)$ which satisfy
$$
\|\varphi\|_{\DD^{M_p,h}_{L^p}}:= \sup_{\alpha\in\NN^n}h^{|\alpha|} \|\partial^{\alpha}\varphi\|_{L^p(\RR^n)}/M_{\alpha}<\infty,
$$
and define the l.c.s.
$$
\DD^{(M_p)}_{L^p}(\RR^n)=\lim_{\substack{\longleftarrow\\ h\rightarrow \infty}} \DD^{M_p,h}_{L^p}(\RR^n)\quad \mbox{and}\quad \DD^{\{M_p\}}_{L^p}(\RR^n)= \lim_{\substack{\longrightarrow \\ h\rightarrow 0^+}} \DD^{M_p,h}_{L^p}(\RR^n).
$$
We use $\DD^*_{L^p}(\RR^n)$ as a common notation for these spaces. The space $\DD^{(M_p)}_{L^p}(\RR^n)$ is a Fr\'echet space, while $\DD^{\{M_p\}}_{L^p}(\RR^n)$ is a complete regular barrelled $(DF)$-space (see, \cite[Section 4.3]{TIBU}). Furthermore, $\SSS^*_{\dagger}(\RR^n)\subseteq \DD^*_{L^p}(\RR^n)$, $1\leq p\leq\infty$, continuously. Throughout the article, we will frequently employ the following result.

\begin{lemma}\label{lemma-for-inc-dl1-f}
Let $\eta$ be a weight of class $*$. Then $\DD^*_{L^1}(\RR^n)\hookrightarrow \mathcal{F}L^1_{\eta}$.
\end{lemma}

\begin{proof} Once we show the continuous inclusion $\DD^*_{L^1}(\RR^n)\subseteq \mathcal{F}L^1_{\eta}$, the density will follow from the density of $\SSS^*_{\dagger}(\RR^n)$ in $\mathcal{F}L^1_{\eta}$. We give the proof only in the Roumieu case, as the Beurling case is similar. Let $h>0$ and $\varphi\in\DD^{M_p,h}_{L^1}(\RR^n)$. Since $|\xi^{\alpha}||\mathcal{F}^{-1}\varphi(\xi)|\leq \|\partial^{\alpha}\varphi\|_{L^1(\RR^n)}$, $\xi\in\RR^n$, $\alpha\in\NN^n$, we have
$$
|\xi|^{2k}|\mathcal{F}^{-1}\varphi(\xi)|^2= \sum_{|\alpha|=k} \frac{k!}{\alpha!}\xi^{2\alpha}|\mathcal{F}^{-1}\varphi(\xi)|^2\leq h^{-2k}M_k^2\|\varphi\|^2_{\DD^{M_p,h}_{L^1}}\sum_{|\alpha|= k} \frac{k!}{\alpha!};
$$
hence $(n^{-1/2}h)^k |\xi|^k|\mathcal{F}^{-1}\varphi(\xi)|/M_k\leq \|\varphi\|_{\DD^{M_p,h}_{L^1}}$, $\xi\in\RR^n$, $k\in\NN$. Consequently,
$$
\eta(\xi)|\mathcal{F}^{-1}\varphi(\xi)|\leq \|\varphi\|_{\DD^{M_p,h}_{L^1}} \eta(\xi)e^{-M(h|\xi|/\sqrt{n})},\quad \xi\in\RR^n.
$$
Since the function on the right is in $L^1(\RR^n)$, this shows that $\DD^{M_p,h}_{L^1}(\RR^n)\subseteq \mathcal{F}L^1_{\eta}$ continuously, and, as $h>0$ is arbitrary, the proof is complete.
\end{proof}

\subsection{Translation-Modulation Invariant Banach Spaces of Ultradistributions and Their Duals}

A Banach space $E$ is said to be a \textit{translation-modulation invariant Banach space of ultradistributions of class $*-\dagger$ on $\RR^n$} (or, a TMIB space of class $*-\dagger$ for short) if it satisfies the following three conditions \cite{DPPV}:
\begin{itemize}
    \item[(I)] $\mathcal{S}^*_{\dagger}(\mathbb{R}^n)\hookrightarrow E\hookrightarrow \mathcal{S}'^*_{\dagger}(\RR^n)$;
    \item[(II)] $T_{x}(E)\subseteq E$ and $M_\xi (E)\subseteq E$, for all $x,\xi\in\mathbb{R}^n$;
\item [(III)] there are $\tau,C>0$ (for every $\tau>0$ there exists $C>0$) such that\footnote{$T_x$ and $M_{\xi}$ are continuous on $E$ in view of (I), (II) and the closed graph theorem.}
    \begin{equation}\label{omega}
   \omega_{E}(x):= \|T_{x}\|_{\mathcal{L}_b(E)}\leq C e^{A(\tau|x|)} \quad  \mbox{and} \quad \nu_{E}(\xi) := \|M_{-\xi}\|_{\mathcal{L}_b(E)}\leq C e^{M(\tau|\xi|)}.
   \end{equation}
\end{itemize}
The functions $\omega_{E},\nu_E:\mathbb{R}^{n}\to (0,\infty)$ are called the \textit{weight functions of the translation and modulation groups} of $E$ respectively. We briefly recall the basic properties of TMIB spaces and refer to \cite[Section 3]{DPPV} for the complete account. Every TMIB space $E$ is separable and its weight functions $\omega_E$ and $\nu_E$ are submultiplicative Borel measurable weights of class $\dagger$ and $*$ respectively; furthermore $\omega_E(0)=\nu_E(0)=1$. For each $e\in E$, the maps $\RR^n\rightarrow E$, $x\rightarrow T_xe$ and $x\mapsto M_xe$, are continuous. The convolution and multiplication on $\SSS^*_{\dagger}(\RR^n)$ uniquely extend to continuous bilinear mappings $*: L^1_{\omega_E}\times E\rightarrow E$ and $\cdot:\mathcal{F}L^1_{\nu_E}\times E\rightarrow E$ which satisfy
\begin{equation}\label{con-mul-module-aaa}
\|f*e\|_E\leq \|f\|_{L^1_{\omega_E}}\|e\|_E\quad\mbox{and}\quad \|g\cdot e\|_E\leq \|g\|_{\mathcal{F}L^1_{\nu_E}}\|e\|_E,
\end{equation}
for all $e\in E$, $f\in L^1_{\omega_E}$, $g\in\mathcal{F}L^1_{\nu_E}$. In particular, $E$ is a Banach convolution module over the Beurling convolution algebra $L^1_{\omega_E}$ and a Banach multiplication module over the Fourier-Beurling multiplication algebra $\mathcal{F}L^1_{\nu_E}$; the multiplication in $\mathcal{F}L^1_{\nu_E}$ is defined via the Fourier transform and the convolution in $L^1_{\nu_E}$ (notice that when $\nu_E$ is bounded from below, $\mathcal{F}L^1_{\nu_E}\subseteq \mathcal{C}_0(\RR^n)$ and the multiplication coincides with the ordinary pointwise multiplication of continuous functions). Furthermore, for $e$, $f$ and $g$ as above
\begin{equation}\label{integ-form-for-con-mult}
f*e=\int_{\RR^n}f(x)T_xe\,dx\quad \mbox{and}\quad g\cdot e=\int_{\RR^n}(\mathcal{F}^{-1}g)(\xi)M_{-\xi}e\,d\xi,
\end{equation}
where the integrals should be interpreted as $E$-valued Bochner integrals. Finally, we point out that the associated Fourier space $\mathcal{F}E$ is a TMIB of class $\dagger-*$. If $\eta$ is a weight of class $\dagger$, $L^p_{\eta}(\RR^n)$, $1\leq p<\infty$, and $\mathcal{C}_{\eta,0}(\RR^n)$ are typical examples of TMIB spaces of class $*-\dagger$ and their associated Fourier spaces are TMIB spaces of class $\dagger-*$.\\
\indent A Banach space $E$ is said to be a \textit{dual translation-modulation invariant Banach space of ultradistributions of class $*-\dagger$ on $\RR^n$} (or, a DTMIB space of class $*-\dagger$ for short) if it is the strong dual of a TMIB space of class $*-\dagger$. If $E$ is a DTMIB space of class $*-\dagger$, then we always have the continuous inclusions $\SSS^*_{\dagger}(\RR^n)\subseteq E\subseteq \SSS'^*_{\dagger}(\RR^n)$, but $\SSS^*_{\dagger}(\RR^n)$ may fail to be dense in $E$ (for example, when $E=L^{\infty}(\RR^n)$). The translation and modulation are continuous operators on $E$, its weight functions $\omega_E$ and $\nu_E$ are defined as in \eqref{omega} and, writing $E=E'_0$ with $E_0$ a TMIB space, it holds that $\omega_E=\check{\omega}_{E_0}$ and $\nu_E=\nu_{E_0}$. For $e\in E$, the maps $\RR^n\rightarrow E$, $x\rightarrow T_xe$ and $x\mapsto M_xe$, are continuous when $E$ is equipped with the weak-$*$ topology. The space $E$ is a Banach convolution module over the Beurling algebra $L^1_{\omega_E}$ and a Banach multiplication module over the Fourier-Beurling algebra $\mathcal{F}L^1_{\nu_E}$, where the convolution and multiplication are defined by duality. Furthermore, \eqref{con-mul-module-aaa} and \eqref{integ-form-for-con-mult} hold true but the integrals in the latter identities should now be interpreted as $E$-valued Pettis integrals with respect to the weak-$*$ topology on $E$. Examples of DTMIB space of class $*-\dagger$ are given by $L^p_{\eta}(\RR^n)$, $1< p\leq\infty$, when $\eta$ is a weight of class $\dagger$; of course their Fourier images are DTMIB spaces of class $\dagger-*$. We point out that all of these facts hold true even if $\{A_p\}_{p\in\NN}$ does not satisfy $(M.2)^*$.\\
\indent One can also consider (D)TMIB spaces of distributions \cite{DPPV} by replacing $\SSS^*_{\dagger}(\RR^n)$ and $\SSS'^*_{\dagger}(\RR^n)$ with $\SSS(\RR^n)$ and $\SSS'(\RR^n)$ respectively and replacing the (sub)exponential bounds in \eqref{omega} with polynomial bounds. All the facts we mentioned above hold true in this case as well (of course, the weight $\eta$ in $L^p_{\eta}(\RR^n)$ and $\mathcal{C}_{\eta,0}(\RR^n)$ now is polynomially bounded).\\
\indent Finally, we point out that if $\{M_p\}_{p\in\NN}$ is non-quasianalytic, the (D)TMIB spaces of class $*-\dagger$ can be viewed as Banach spaces of ultradistributions having two module structures in the sense of \cite{brfe83} (cf. \cite{fei84}). However, when $\{M_p\}_{p\in\NN}$ is (quasi)analytic this is no longer the case, since in this case a (D)TMIB space of class $*-\dagger$ may contain only (quasi)analytic functions (for example $\mathcal{F}L^1_{\exp(k|\cdot|)}$, $k>0$, which is a TMIB space of class $(p!)-(p!)$).

\section{Amalgam spaces}\label{amalgam sect}

\subsection{Uniformly concentrated partitions of unity}

The key ingredient in the discrete characterisation of the amalgam spaces that we are going to show later in the article are partitions of unity with functions in $\SSS^*_{\dagger}(\RR^n)$ that have uniform bounds on their growth. They will also be useful in the proofs of several auxiliary results.

\begin{definition}\label{def-bupuultr-for-aml}
Let $\Lambda$ be a countable index set. We say that the family of functions $\Psi=\{\psi_{\lambda}\}_{\lambda\in\Lambda}\subseteq \SSS^*_{\dagger}(\RR^n)$ is a \textit{uniformly concentrated partition of unity of class $*-\dagger$} (from now, abbreviated as UCPU of class $*-\dagger$) if there is a sequence of points $\{y_{\lambda}\}_{\lambda\in\Lambda}\subseteq \RR^n$ such that:
\begin{itemize}
      \item[$(1)$] for every $h>0$ (resp., for some $h>0$)
    \begin{equation}
        \sup_{\lambda\in\Lambda}\sup_{\alpha\in\NN^n}\sup_{x\in\mathbb{R}^n} \frac{h^\alpha |\partial^{\alpha} \psi_{\lambda}(x) |e^{A(h|x-y_\lambda|)}} {M_{\alpha}}<\infty;
    \end{equation}
    \item[$(2)$] for any compact set $K\subseteq\mathbb{R}^n$, there is $C_K>0$ such that
        $$
        \sup_{x\in\RR^n}|\{\lambda\in \Lambda\,|\,x\in y_\lambda+K\}|\leq C_K;
        $$
    \item[$(3)$] there is an open bounded neighbourhood $U$ of the origin in $\RR^n$ such that $\mathbb{R}^n=\cup_{\lambda\in\Lambda}(y_\lambda+U)$;
    \item[$(4)$] $\sum_{\lambda\in\Lambda}\psi_\lambda(x)=1$, $x\in \mathbb{R}^n$.
\end{itemize}
\end{definition}

The summation in $(4)$ always makes sense. To see this, we first record the following fact. It can be shown as in \cite[Theorem 22]{feic111}, however, for the sake of completeness we give a direct proof.

\begin{lemma}\label{lemma-for-familyof-points-b}
Let $\Lambda$ be a countable index set and let $\{y_{\lambda}\}_{\lambda\in\Lambda}$ be a family of points in $\RR^n$ which satisfies the condition $(2)$. Then, for every $\varepsilon>0$
$$
\sup_{\mu\in\Lambda}\sum_{\lambda\in\Lambda} (1+|y_{\lambda}-y_{\mu}|)^{-n-\varepsilon}<\infty.
$$
\end{lemma}

\begin{proof} Set $K:=\{x\in\RR^n\,|\, |x|\leq 1\}$ and let $C_K>0$ be the constant from the condition $(2)$ for $K$; hence $\sum_{\lambda\in\Lambda} \mathbf{1}_{y_{\lambda}+K}(x)\leq C_K$, $\forall x\in\RR^n$. For $\mu\in\Lambda$, we have
\begin{align*}
\sum_{\lambda\in\Lambda} \frac{1}{(1+|y_{\lambda}-y_{\mu}|)^{n+\varepsilon}}&= \frac{1}{|K|} \sum_{\lambda\in\Lambda} \int_{y_{\lambda}+K}\frac{dx}{(1+|y_{\lambda}-y_{\mu}|)^{n+\varepsilon}}\\
&\leq \frac{2^{n+\varepsilon}}{|K|} \sum_{\lambda\in\Lambda} \int_{y_{\lambda}+K}\frac{dx}{(1+|x-y_{\mu}|)^{n+\varepsilon}}\\
&= \frac{2^{n+\varepsilon}}{|K|}  \int_{\RR^n}\frac{1}{(1+|x-y_{\mu}|)^{n+\varepsilon}} \sum_{\lambda\in\Lambda}\mathbf{1}_{y_{\lambda}+K}(x)dx\\
&\leq 2^{n+\varepsilon}C_K\|(1+|\cdot|)^{-n-\varepsilon}\|_{L^1(\RR^n)}/|K|.
\end{align*}
\end{proof}

Employing this lemma, a straightforward computation shows that if $\{\psi_{\lambda}\}_{\lambda\in\Lambda}$ and $\{y_{\lambda}\}_{\lambda\in\Lambda}$ satisfy $(1)-(3)$, then the series $\sum_{\lambda\in\Lambda}\psi_\lambda$ is absolutely summable in $\DD^*_{L^{\infty}_{(1+|\cdot|)^{-n-1}}}(\RR^n)$. Consequently, if $\{\psi_{\lambda}\}_{\lambda\in\Lambda}$ is a UCPU of class $*-\dagger$, then for any $f\in\SSS'^*_{\dagger}(\RR^n)$, the series $\sum_{\lambda\in\Lambda}\psi_{\lambda}f$ is absolutely summable to $f$ in $\SSS'^*_{\dagger}(\RR^n)$.

\begin{remark}\label{boun-bupuforlealge-nor}
Notice that the condition $(1)$ is equivalent to the following: the set $\{T_{-y_{\lambda}}\psi_{\lambda}\,|\, \lambda\in\Lambda\}$ is a bounded subset of $\SSS^*_{\dagger}(\RR^n)$. Consequently, $\sup_{\lambda\in\Lambda}\|\psi_\lambda\|_{\mathcal{F}L^1_{\eta}} <\infty$ for any weight $\eta$ of class $*$.
\end{remark}

\begin{remark}\label{rem-bupu-of-class-a-dss}
A typical example of a UCPU of class $*-\dagger$ is to take $\Lambda=\ZZ^n$, $y_{\mathbf{n}}=\mathbf{n}$ and $\psi_{\mathbf{n}}:=(T_{\mathbf{n}}\phi)*\varphi$, $\mathbf{n}\in\ZZ^n$, where $\phi\in\DD(\RR^n)$ is such that $\supp\phi\subseteq (-1,1)^n$, $\sum_{\mathbf{n}\in\ZZ^n}T_{\mathbf{n}}\phi(x)=1$, and $\varphi$ is any element of $\SSS^*_{\dagger}(\RR^n)$ that satisfies $\int_{\RR^n}\varphi(x)dx=1$ (for example, a normalised Gaussian). Notice that $\psi_{\mathbf{n}}=T_{\mathbf{n}}\psi$ with $\psi:=\phi*\varphi$. Similar constructions can be performed when $\Lambda$ is an arbitrary lattice in $\RR^n$.
\end{remark}

\begin{remark}
The definition of UCPU is inspired by the bounded uniform partitions of unity (abbreviated as BUPU) introduced by Feichtinger \cite{feich} in the distributional setting (cf. \cite{f-p-p,lep-m}) and in the non-quasianalytic setting. In the distributional setting, a BUPU is a family of functions $\psi_{\lambda}$, $\lambda\in\Lambda$, belonging to $\DD(\RR^n)$ which satisfy the conditions $(2)$ and $(4)$, while the condition $(1)$ is replaced with the following two
\begin{itemize}
\item[$(0)'$] there is an open bounded neighbourhood $V$ of the origin in $\RR^n$ such that $\supp\psi_{\lambda}\subseteq y_{\lambda}+V$, $\lambda\in\Lambda$;
\item[$(1)'$] $\sup_{\lambda\in\Lambda}\|\psi_{\lambda}\|_{\mathcal{F}L^1_{\nu}}<\infty$, where $\nu$ is a measurable submultiplicative polynomially bounded weight (usually chosen appropriately for the intended purpose).
\end{itemize}
Because of $(0)'$ and $(2)$, the series in $(4)$ is locally finite which, in turn, implies that $(3)$ is automatically satisfied. When $\{M_p\}_{p\in\NN}$ is non-quasianalytic, the definition of a BUPU remains the same with only one change: now one requires that $\psi_{\lambda}\in\DD^*(\RR^n)$ \cite{Komatsu1}, for all $\lambda\in\Lambda$. However, when $\{M_p\}_{p\in\NN}$ is quasianalytic, this is no longer applicable as there are no non-trivial functions with compact support in $\SSS^*_{\dagger}(\RR^n)$. We have chosen to change the terminology to UCPU to avoid confusion, since the term ``uniform'' in BUPU indicates that the supports of the functions have uniform size.
\end{remark}

\begin{remark}\label{bupufor-s-dis}
Similarly, one can consider UCPUs in the distributional setting. The definition is the same with the only difference being that now $\psi_{\lambda}\in\SSS(\RR^n)$, $\lambda\in\Lambda$, and the condition $(1)$ is replaced with
\begin{itemize}
\item[$(1)''$] $\sup_{\lambda\in\Lambda}\sup_{|\alpha|\leq N}\sup_{x\in\RR^n} |\partial^{\alpha}\psi_{\lambda}(x)|(1+|x-y_{\lambda}|)^N<\infty$, for all $N\in\NN$.
\end{itemize}
\end{remark}

We are going to need the following technical result.

\begin{lemma}\label{est-partial-growt-fourbeuralg}
Let $\eta$ be a weight of class $*$ and let $K$ be a compact subset of $\RR^n$. Let $\chi\in\SSS^*_{\dagger}(\RR^n)$ and  $\Psi=\{\psi_{\lambda}\}_{\lambda\in\Lambda}$ be a UCPU of class $*-\dagger$ with points $\{y_{\lambda}\}_{\lambda\in\Lambda}$. In the Beurling case, for every $\tilde{h}>0$ there exists $\tilde{C}\geq 1$ such that
\begin{equation}\label{est-fourleb-alg-growth}
\sup_{x\in y_{\mu}+K}\|\psi_{\lambda}T_x\chi\|_{\mathcal{F}L^1_{\eta}}\leq \tilde{C}e^{-A(\tilde{h}|y_{\lambda}-y_{\mu}|)},\quad \mbox{for all}\,\, \lambda,\mu\in\Lambda.
\end{equation}
In the Roumieu case, there are $\tilde{h}>0$ and $\tilde{C}\geq 1$ such that \eqref{est-fourleb-alg-growth} holds true.
\end{lemma}

\begin{proof} There are $C,\tau>0$ (resp., for every $\tau>0$ there is $C>0$) such that $\eta(x)\leq Ce^{M(\tau|x|)}$, $x\in\RR^n$. In the Beurling case, let $h>0$ be arbitrary but fixed; without loss in generality, we can assume $h\geq H\tau\sqrt{n}$. In the Roumieu case, choose $h>0$ such that $\chi\in\SSS^{M_p,2h}_{A_p,2h}$ and the condition $(1)$ on the UCPU $\Psi$ holds true with $2h$ and then pick $\tau>0$ such that $\tau\leq h/(H\sqrt{n})$. Employing the same technique as in the proof of Lemma \ref{lemma-for-inc-dl1-f}, we infer
\begin{equation}
|\mathcal{F}^{-1}(\psi_{\lambda}T_x\chi)(\xi)|e^{M(h|\xi|/\sqrt{n})}\leq \|\psi_{\lambda}T_x\chi\|_{\DD^{M_p,h}_{L^1}},\quad \xi\in\RR^n.
\end{equation}
Consequently,
\begin{equation}\label{ine-four-intultradif}
\|\psi_{\lambda}T_x\chi\|_{\mathcal{F}L^1_{\eta}}\leq C_1 \|e^{-M(\tau|\cdot|)}e^{M(\tau H|\cdot|)} \mathcal{F}^{-1}(\psi_{\lambda}T_x\chi)\|_{L^1}\leq C_2 \|\psi_{\lambda}T_x\chi\|_{\DD^{M_p,h}_{L^1}}.
\end{equation}
For $\lambda,\mu\in\Lambda$ and $x\in y_{\mu}+K$, we have
\begin{align*}
h^{|\alpha|}&\|\partial^{\alpha}(\psi_{\lambda}T_x\chi)\|_{L^1(\RR^n)}/M_{\alpha}\\
& \leq \frac{C_3}{2^{|\alpha|}}\sum_{\beta\leq \alpha} {\alpha\choose\beta} \int_{\RR^n}\frac{(2h)^{|\alpha|}|\partial^{\beta}\psi_{\lambda}(y)|e^{A(h|y-y_{\lambda}|)} |\partial^{\alpha-\beta}\chi(y-x)| e^{A(h|y-x|)}}{M_{\alpha} e^{A(h|y-y_{\lambda}|)} e^{A(h|y-x|/H)} e^{A(h|y-x|/H)}}dy\\
&\leq C_4 \sup_{y\in\RR^n} e^{-A(h|y-y_{\lambda}|)} e^{-A(h|y-x|/H)}\\
&\leq 2C_4e^{-A(h|y_{\lambda}-y_{\mu}|/(4H))} \sup_{y\in\RR^n} e^{-A(h|y-y_{\lambda}|)}e^{A(h|y_{\lambda}-y|/(2H))} e^{-A(h|y-x|/H)}e^{A(h|y-y_{\mu}|/(2H))}\\
&\leq 4C_4e^{-A(h|y_{\lambda}-y_{\mu}|/(4H))}e^{A(h|x-y_{\mu}|/H)}\leq C_5 e^{-A(h|y_{\lambda}-y_{\mu}|/(4H))}.
\end{align*}
The claim in the Lemma immediately follows once we employ these bounds in \eqref{ine-four-intultradif}.
\end{proof}

The proof of the following result is analogous and we omit it (cf. Lemma \ref{lemma-for-inc-dl1-f}).

\begin{lemma}\label{cor-for-dis-betw-two-elebupu}
Let $\eta$ be a weight of class $*$ and let $\{\psi_{\lambda}\}_{\lambda\in\Lambda}$ be a UCPU of class $*-\dagger$. Then for every $h>0$ (resp., there is $h>0$) such that
$$
\sup_{\lambda,\mu\in\Lambda}e^{A(h|y_{\lambda}-y_{\mu}|)} \|\psi_{\lambda}\psi_{\mu}\|_{\mathcal{F}L^1_{\eta}}<\infty.
$$
\end{lemma}

\subsection{The main results}

For the definition of the amalgam spaces having (D)TMIB as a local component, we need several technical results about the following class of subspaces of $\SSS'^*_{\dagger}(\RR^n)$. Let $E$ be a TMIB or a DTMIB space of class $*-\dagger$. Set
$$
E_{\operatorname{loc}}:=\{f\in\SSS'^*_{\dagger}(\RR^n)\,|\, f\chi\in E,\, \forall \chi\in\SSS^*_{\dagger}(\RR^n)\}.
$$
We equip $E_{\operatorname{loc}}$ with the topology given by the system of seminorms
$$
E_{\operatorname{loc}}\rightarrow[0,\infty),\, f\mapsto\|f\chi\|_E,\quad \chi\in\SSS^*_{\dagger}(\RR^n).
$$

\begin{remark}\label{rem-for-dif-tmibspacdtspa}
A (D)TMIB space $E$ of class $*-\dagger$ can also be a (D)TMIB space for other pair of sequences $\{M_p\}_{p\in\NN}$ and $\{A_p\}_{p\in\NN}$. In fact, every (D)TMIB space of class $*-\dagger$ is also a (D)TMIB space of class $(p!)-(p!)$. One can show that $E_{\operatorname{loc}}$ does not depend on $\{M_p\}_{p\in\NN}$ but, in general, it does depend on $\{A_p\}_{p\in\NN}$. However, we will not need this fact, since $E_{\operatorname{loc}}$ is only an auxiliary space that will help us to define the amalgam spaces. At the very end, we will show that the amalgam spaces do not depend on either $\{M_p\}_{p\in\NN}$ nor $\{A_p\}_{p\in\NN}$ (see Proposition \ref{depe-donton-seamalgams} and Remark \ref{rem-for-amalg-spac-not-deponseq} below).
\end{remark}

\begin{lemma}\label{lemma-locspa-continc}
The space $E_{\operatorname{loc}}$ is a Hausdorff complete l.c.s. Furthermore $E\rightarrow E_{\operatorname{loc}}\hookrightarrow \SSS'^*_{\dagger}(\RR^n)$.
\end{lemma}

\begin{proof} The fact that $E_{\operatorname{loc}}$ is Hausdorff immediately follows from the decomposition $f=\sum_{\lambda\in\Lambda} f\psi_{\lambda}$, $f\in E_{\operatorname{loc}}$, where $\{\psi_{\lambda}\}_{\lambda\in\Lambda}$ is any UCPU of class $*-\dagger$. Clearly $E\subseteq E_{\operatorname{loc}}\subseteq \SSS'^*_{\dagger}(\RR^n)$. The first inclusion is clearly continuous. To prove the continuity of the second inclusion, let $B$ be a bounded subset of $\SSS^*_{\dagger}(\RR^d)$. We employ Proposition \ref{facto-s-mul} to find $\chi\in\SSS^*_{\dagger}(\RR^n)$  and a bounded subset $B_1$ of $\SSS^*_{\dagger}(\RR^n)$ such that $B=\chi B_1$. If $E$ is a TMIB, for $f\in E_{\operatorname{loc}}$, we infer
$$
\sup_{\varphi\in B}|\langle f,\varphi\rangle|=\sup_{\psi\in B_1}|\langle f\chi,\psi\rangle|\leq \|f\chi\|_E\sup_{\psi\in B_1}\|\psi\|_{E'},
$$
which proves the desired continuity; the case when $E$ is a DTMIB is analogous. The density of $E_{\operatorname{loc}}$ in $\SSS'^*_{\dagger}(\RR^n)$ follows from the density of $\SSS^*_{\dagger}(\RR^n)$ in the latter space. The completeness of $E_{\operatorname{loc}}$ can be easily deduced form the continuity of the inclusion $E_{\operatorname{loc}}\subseteq \SSS'^*_{\dagger}(\RR^n)$.
\end{proof}

\begin{lemma}\label{lemma-reg-e-loc-proper}
Let $B$ be a non-empty bounded subset of $\SSS^*_{\dagger}(\RR^n)$.
\begin{itemize}
\item[$(i)$] There exist $\chi\in\SSS^*_{\dagger}(\RR^n)$ and $C\geq 1$ such that $\sup_{\varphi\in B}\|f\varphi\|_E\leq C\|f\chi\|_E$, for all $f\in E_{\operatorname{loc}}$.
\item[$(ii)$] For any $f\in E_{\operatorname{loc}}$, the mapping $\SSS^*_{\dagger}(\RR^n)\rightarrow E$, $\varphi\mapsto f\varphi$, is continuous.
\item[$(iii)$] For any $\varphi\in\SSS^*_{\dagger}(\RR^d)$ and $f\in E_{\operatorname{loc}}$, the mapping $\RR^n\rightarrow E$, $x\mapsto fT_x\varphi$, is continuous.
\item[$(iv)$] For any $f\in E_{\operatorname{loc}}$, the function $\RR^n\rightarrow [0,\infty)$, $x\mapsto \sup_{\varphi\in B}\|fT_x\varphi\|_E$, is Borel measurable.
\end{itemize}
\end{lemma}

\begin{proof} To prove $(i)$, we apply Proposition \ref{facto-s-mul} to infer the existence of $\chi\in \SSS^*_{\dagger}(\RR^n)$ and a bounded subset $B_1$ of $\SSS^*_{\dagger}(\RR^n)$ such that $B=\chi B_1$. Then
$$
\sup_{\varphi\in B}\|f\varphi\|_E=\sup_{\psi\in B_1}\|f\chi\psi\|_E\leq \|f\chi\|_E\sup_{\psi\in B_1}\|\psi\|_{\mathcal{F}L^1_{\nu_E}},
$$
which verifies $(i)$. Since $\SSS^*_{\dagger}(\RR^n)$ is bornological, to prove $(ii)$ it suffices to show that this mapping maps bounded sets into bounded sets. The latter follows from $(i)$.\\
\indent We turn our attention to $(iii)$. Let $x_0\in \RR^n$ be arbitrary but fixed. We employ Proposition \ref{facto-s-mul} to find $\varphi_1,\varphi_2\in\SSS^*_{\dagger}(\RR^n)$ such that $\varphi=\varphi_1\varphi_2$. Clearly, $\{T_x\varphi_1\,|\, |x-x_0|\leq 1\}$ is a bounded subset of $\SSS^*_{\dagger}(\RR^n)$ and hence $(i)$ implies that $C_1:=\sup_{|x-x_0|\leq 1}\|fT_x\varphi_1\|_E<\infty$. For $x\in \RR^n$ satisfying $|x-x_0|\leq1$, we estimate as follows
\begin{align*}
\|fT_x\varphi-fT_{x_0}\varphi\|_E&\leq \|fT_x\varphi_1(T_x\varphi_2-T_{x_0}\varphi_2)\|_E +\|fT_{x_0}\varphi_2(T_x\varphi_1-T_{x_0}\varphi_1)\|_E\\
&\leq C_1\|T_x\varphi_2-T_{x_0}\varphi_2\|_{\mathcal{F}L^1_{\nu_E}}+ \|fT_{x_0}\varphi_2\|_E\|T_x\varphi_1-T_{x_0}\varphi_1\|_{\mathcal{F}L^1_{\nu_E}},
\end{align*}
which implies the validity of $(iii)$.\\
\indent It remains to prove $(iv)$. Since $\sup_{\varphi\in B}\|fT_x\varphi\|_E=\sup_{\varphi\in \overline{B}}\|fT_x\varphi\|_E$ (in view of $(ii)$), we can assume that $B$ is closed. We claim that $B$ is a compact metrisable space when equipped with the induced topology from $\SSS^*_{\dagger}(\RR^n)$. The compactness follows from the fact that $\SSS^*_{\dagger}(\RR^n)$ is Montel. This proves the claim in the Beurling case. The metrisability in the Roumieu case follows from \cite[Theorem 1.7, p. 128]{Sch} since $\SSS^{\{M_p\}}_{\{A_p\}}(\RR^n)$ is the strong dual of the separable Montel space $\SSS'^{\{M_p\}}_{\{A_p\}}(\RR^n)$ and the weak and strong topologies coincide on the bounded subsets of $\SSS^{\{M_p\}}_{\{A_p\}}(\RR^n)$. Hence, there exists a countable dense subset of $B$. In view of $(iii)$, the claim in $(iv)$ follows.
\end{proof}

We are now ready to define the amalgam spaces with (D)TMIB of class $*-\dagger$ as a local component. Let $\chi\in\SSS^*_{\dagger}(\RR^n)\backslash\{0\}$, let $\eta$ be a weight of class $\dagger$ and let $E$ be a TMIB or a DTMIB space of class $*-\dagger$. In view of Lemma \ref{lemma-reg-e-loc-proper} $(iii)$, for any $f\in E_{\operatorname{loc}}$ the function $\RR^n\rightarrow [0,\infty)$, $x\mapsto \|fT_x\chi\|_E$, is continuous so we can take its $L^p_{\eta}$-norm, $1\leq p\leq\infty$, as well as its $\mathcal{C}_{\eta,0}$-norm; of course these may be infinite. We define the amalgam space $W(E,L^p_{\eta})$, $1\leq p\leq \infty$, to be the space of all $f\in E_{\operatorname{loc}}$ such that
\begin{equation}\label{norm-lpbspp}
\|f\|_{p,\chi,\eta}:=\left(\int_{\RR^n} \|fT_x\chi\|_E^p\eta(x)^pdx\right)^{1/p}<\infty
\end{equation}
with the obvious modification when $p=\infty$. Clearly, $\|\cdot\|_{p,\chi,\eta}$ is a seminorm on $W(E,L^p_{\eta})$, $1\leq p\leq \infty$; a priori it is not clear that it is a norm. Similarly, we define $W(E,\mathcal{C}_{\eta,0})$ as the space of all $f\in E_{\operatorname{loc}}$ such that the function $\RR^n\rightarrow [0,\infty)$, $x\mapsto \|fT_x\chi\|_E$ belongs to $\mathcal{C}_{\eta,0}(\RR^n)$. Clearly, $\|\cdot\|_{\infty,\chi,\eta}$ is a seminorm on this space.

\begin{proposition}\label{lemma-amalg-space-alter-dfnts}
${}$
\begin{itemize}
\item[$(i)$] The definitions of $W(E,L^p_{\eta})$, $1\leq p\leq\infty$, and $W(E,\mathcal{C}_{\eta,0})$ do not depend on the choice of $\chi\in\SSS^*_{\dagger}(\RR^n)\backslash\{0\}$. For each $\chi\in\SSS^*_{\dagger}(\RR^n)\backslash\{0\}$, \eqref{norm-lpbspp} is a norm on $W(E,L^p_{\eta})$, $1\leq p\leq \infty$, and \eqref{norm-lpbspp} with $p=\infty$ is a norm on $W(E,\mathcal{C}_{\eta,0})$. Furthermore, different choices of $\chi\in\SSS^*_{\dagger}(\RR^n)\backslash\{0\}$ produce equivalent norms. Moreover, $W(E,L^p_{\eta})$, $1\leq p\leq\infty$, and $W(E,\mathcal{C}_{\eta,0})$ are Banach spaces and they are continuously included into $E_{\operatorname{loc}}$.
\item[$(ii)$] Let $B$ be a bounded subset of $\SSS^*_{\dagger}(\RR^n)$ such that $B\backslash\{0\}\neq \emptyset$. For any $f\in E_{\operatorname{loc}}$ and $1\leq p\leq \infty$, $f\in W(E,L^p_{\eta})$ if and only if the function $\RR^n\rightarrow [0,\infty)$, $x\mapsto \sup_{\varphi\in B}\|fT_x\varphi\|_E$, belongs to $L^p_{\eta}(\RR^n)$. Furthermore,
\begin{equation}\label{norm-lpama-boundsubset}
\|f\|_{p,B,\eta}:=\left(\int_{\RR^n}\sup_{\varphi\in B}\|fT_x\varphi\|_E^p\eta(x)^pdx\right)^{1/p}
\end{equation}
is a norm on $W(E,L^p_{\eta})$ equivalent to the norm \eqref{norm-lpbspp}. For any $f\in E_{\operatorname{loc}}$, $f\in W(E,\mathcal{C}_{\eta,0})$ if and only if the function $\RR^n\rightarrow [0,\infty)$, $x\mapsto \sup_{\varphi\in B}\|fT_x\varphi\|_E$, belongs to $L^{\infty}_{\eta,0}(\RR^n)$ and the initial norm on $W(E,\mathcal{C}_{\eta,0})$ is equivalent to the norm \eqref{norm-lpama-boundsubset} with $p=\infty$.
\end{itemize}
\end{proposition}

\begin{remark}
In view of Lemma \ref{lemma-reg-e-loc-proper} $(iv)$, it makes sense to compute \eqref{norm-lpama-boundsubset} for any $f\in E_{\operatorname{loc}}$; of course, this quantity may be infinite.
\end{remark}

\begin{proof} We first address $(i)$. For the moment, denote by $X_{p,\chi}$, $1\leq p\leq \infty$, the space $W(E, L^p_{\eta})$ defined by $\chi\in\SSS^*_{\dagger}(\RR^n)\backslash\{0\}$. Let $\chi,\psi\in \SSS^*_{\dagger}(\RR^n)\backslash\{0\}$ be arbitrary but fixed and let $f\in X_{p,\psi}$. In view of Lemma \ref{lemma-reg-e-loc-proper} $(iii)$, the mapping
\begin{equation}\label{fun-boc-int-rep-f}
\RR^{2n}\rightarrow E,\quad (x,y)\mapsto fT_x\chi T_y\psi T_y\overline{\psi},
\end{equation}
is continuous. We claim that it is Bochner integrable on $\RR^n_y$ for each fixed $x\in\RR^n$. To see this, denote by $q$ the H\"older conjugate index to $p$ and estimate as follows
\begin{align}
\int_{\RR^n}\|fT_x\chi T_y\psi T_y\overline{\psi}\|_Edy&\leq \int_{\RR^n}\|fT_y\psi\|_E\|T_y(\overline{\psi} T_{x-y}\chi)\|_{\mathcal{F}L^1_{\nu_E}}dy\label{est-for-equ-of-the-func}\\
&\leq \|f\|_{p,\psi,\eta} \left(\int_{\RR^n}\frac{\|\overline{\psi}T_{x-y}\chi\|_{\mathcal{F}L^1_{\nu_E}}^q} {\eta(y)^q}dy\right)^{1/q} \nonumber\\
&\leq \frac{C\|f\|_{p,\psi,\eta}}{\eta(x)}\left(\int_{\RR^n}\|\overline{\psi} T_y\chi\|_{\mathcal{F}L^1_{\nu_E}}^qe^{qA(\tau|y|)}dy\right)^{1/q}, \label{equ-for-inc-eloc-aml-spa-nectr}
\end{align}
for some $C,\tau>0$ (resp., for every $\tau>0$ and a corresponding $C>0$), with the obvious modification when $q=\infty$ (i.e., $p=1$). Employing Lemma \ref{lemma-for-inc-dl1-f}, it is straightforward to verify that the very last integral (or supremum over $\RR^n_y$ when $q=\infty$) is finite in the Beurling case and that one can find $\tau>0$ such that it is finite in the Roumieu case. This shows the Bochner integrability of \eqref{fun-boc-int-rep-f} over $\RR^n_y$ for each fixed $x\in\RR^n$. For each $x\in\RR^n$, we claim that
\begin{equation}\label{ide-for-e-funwitht}
\|\psi\|^2_{L^2(\RR^n)}fT_x\chi=e_x,\quad \mbox{where}\quad e_x:= \int_{\RR^n} fT_x\chi T_y\psi T_y\overline{\psi}dy\in E.
\end{equation}
Let $\varphi\in\SSS^*_{\dagger}(\RR^n)$. Since the function $\RR^{2n}\rightarrow \CC$, $(y,\xi)\mapsto \varphi(\xi)\psi(\xi-y)\overline{\psi(\xi-y)}$, belongs to $\SSS^*_{\dagger}(\RR^{2n})$, we infer
\begin{align*}
\langle e_x,\varphi\rangle&=\int_{\RR^n}\langle fT_x\chi T_y\psi T_y\overline{\psi},\varphi\rangle dy=\langle \mathbf{1}_{\RR^n}(y)\otimes (f(\xi)T_x\chi(\xi)),\varphi(\xi)\psi(\xi-y)\overline{\psi(\xi-y)}\rangle\\
&=\left\langle f(\xi)T_x\chi(\xi),\varphi(\xi)\int_{\RR^n} \psi(\xi-y)\overline{\psi(\xi-y)}dy\right\rangle=\|\psi\|^2_{L^2(\RR^n)}\langle fT_x\chi,\varphi\rangle,
\end{align*}
which proves \eqref{ide-for-e-funwitht}. We employ \eqref{est-for-equ-of-the-func} and the Minkowski integral inequality to infer
\begin{align}
\|f\|_{p,\chi,\eta}&\leq \|\psi\|^{-2}_{L^2(\RR^n)}\int_{\RR^n}\left(\int_{\RR^n}\|fT_{x-y}\psi\|_E^p \|\overline{\psi}T_y\chi\|_{\mathcal{F}L^1_{\nu_E}}^p\eta(x)^pdx\right)^{1/p}dy\nonumber\\
&\leq C_1\|f\|_{p,\psi,\eta};\label{norm-ine-strsv}
\end{align}
the case $p=\infty$ is analogous. This proves that the definition of $X_{p,\chi}$ does not depend on $\chi\in\SSS^*_{\dagger}(\RR^n)\backslash\{0\}$. It also proves that if $\|f\|_{p,\chi,\eta}=0$ for some $\chi\in\SSS^*_{\dagger}(\RR^n)\backslash\{0\}$ than this holds for all $\chi\in\SSS^*_{\dagger}(\RR^n)$, which implies that $f\chi=0$ for all $\chi\in\SSS^*_{\dagger}(\RR^n)$ and hence $f=0$ (cf. the proof of Lemma \ref{lemma-locspa-continc}). Consequently $\|\cdot\|_{p,\chi,\eta}$ is a norm for each $\chi\in\SSS^*_{\dagger}(\RR^n)\backslash\{0\}$ and \eqref{norm-ine-strsv} now proves that different choices of $\chi$ produce equivalent norms. Specialising \eqref{equ-for-inc-eloc-aml-spa-nectr} for $x=0$, we deduce $\|f\chi\|_E\leq C'_1\|f\|_{p,\psi,\eta}$ which implies that the inclusion $W(E,L^p_{\eta})\subseteq E_{\operatorname{loc}}$ is continuous. This fact together with the completeness of $E_{\operatorname{loc}}$ and Fatou's lemma imply that $W(E,L^p_{\eta})$ is complete (the case $p=\infty$ is trivial). Next, we prove that $W(E,\mathcal{C}_{\eta,0})$ is independent of $\chi$. Without loss of generality, we can assume that $\eta$ is continuous. Let $\chi,\psi\in\SSS^*_{\dagger}(\RR^n)\backslash\{0\}$ and let $f\in W(E,\mathcal{C}_{\eta,0})$ as defined by $\psi$. Then $f\in W(E,L^{\infty}_{\eta})$ and, in view of the above, both $\|f\|_{\infty,\chi,\eta}$ and $\|f\|_{\infty,\chi,\eta}$ are finite. Employing \eqref{est-for-equ-of-the-func}, we infer
\begin{equation}\label{L-infty ocena}
    \|fT_x\chi\|_E\eta(x)\leq C\int_{\RR^n}\|fT_{x-y}\psi\|_E\eta(x-y) \|\overline{\psi}T_y\chi\|_{\mathcal{F}L^1_{\nu_E}}e^{A(\tau|y|)}dy,\quad x\in\RR^n,
\end{equation}
for some $C,\tau>0$ in the Beurling case and for every $\tau>0$ and a corresponding $C>0$ in the Roumieu case. The integral $\int_{\RR^n}\|\overline{\psi}T_y\chi\|_{\mathcal{F}L^1_{\nu_E}}e^{A(\tau|y|)}dy$ is finite in the Beurling case and we can choose $\tau>0$ so that it is finite in the Roumieu case. Since $\sup_{x,y\in\RR^n}\|fT_{x-y}\psi\|_E\eta(x-y) \leq \|f\|_{\infty,\psi,\eta}$, dominated convergence implies that $\|fT_x\chi\|_E\eta(x)\rightarrow 0$, as $|x|\rightarrow\infty$, which yields that $W(E,\mathcal{C}_{\eta,0})$ is independent from the choice of $\chi\in\SSS^*_{\dagger}(\RR^n)\backslash\{0\}$. Since $W(E,\mathcal{C}_{\eta,0})$ is topologically imbedded into $W(E,L^{\infty}_{\eta})$, its completeness will follow once we show that it is a closed subspace of the latter. The proof of this fact is straightforward and we omit it.\\
\indent The validity of $(ii)$ follows from the fact that each bounded subset $B$ of $\SSS^*_{\dagger}(\RR^n)$ can be multiplicatively split as $B=\chi B_1$ with some $\chi\in\SSS^*_{\dagger}(\RR^n)$ and a bounded subset $B_1$ of $\SSS^*_{\dagger}(\RR^n)$ in view of Proposition \ref{facto-s-mul} (notice that $\sup_{\varphi\in B}\|fT_x\varphi\|_E\leq \|fT_x\chi\|_E\sup_{\psi\in B_1}\|\psi\|_{\mathcal{F}L^1_{\nu_E}}$, for all $f\in E_{\operatorname{loc}}$, $x\in\RR^n$).
\end{proof}

\begin{remark}\label{rem-amalgam-for-dis-withs-fun}
By employing the same technique, one can show all the claims in Proposition \ref{lemma-amalg-space-alter-dfnts} when $E$ is a (D)TMIB space of distributions and $\eta$ a polynomially bounded weight (i.e., moderate with respect to the Beurling weight $(1+|\cdot|)^{\tau}$, for some $\tau\geq 0$). In this case, instead of Proposition \ref{facto-s-mul}, one employs the analogous fact for $\SSS(\RR^n)$ \cite[Theorem 3.2]{Voigt} (cf. \cite{D-M,pet-vr}). This gives an alternative proof of the widely known fact that one can use windows without compact support to define the amalgam spaces in the distributional setting.
\end{remark}

Our goal is to provide a discrete characterisation via UCPUs of the amalgam spaces we defined above. We refer to \cite[Theorem 2]{feich} for such characterisation in the distributional and in the non-quasianalytic ultradistributional setting; however, the techniques employed in the proof of this result are not applicable when $\{M_p\}_{p\in\NN}$ is quasianalytic because of the lack of functions with compact support in $\SSS^*_{\dagger}(\RR^n)$.\\
\indent Let $\eta$ be a weight of class $\dagger$ and let $\Psi=\{\psi_{\lambda}\}_{\lambda\in \Lambda}$ be a UCPU of class $*-\dagger$ with points $\{y_{\lambda}\}_{\lambda\in\Lambda}$. We define $W_{\Psi}(E,L^p_{\eta})$, $1\leq p\leq\infty$, to be the space of all $f\in E_{\operatorname{loc}}$ such that
\begin{equation}\label{norm-amalgam-space-ultrad}
\|f\|_{p,\Psi,\eta}:=\|\{\|f\psi_{\lambda}\|_E \eta(y_{\lambda})\}_{\lambda\in\Lambda}\|_{\ell^p(\Lambda)}= \left(\sum_{\lambda\in\Lambda}\|f\psi_{\lambda}\|_E^p \eta(y_{\lambda})^p\right)^{1/p}<\infty
\end{equation}
with the obvious modification when $p=\infty$ and we equip it with the norm \eqref{norm-amalgam-space-ultrad}. Similarly, we define $W_{\Psi}(E,\mathcal{C}_{\eta,0})$ to be the space of all $f\in E_{\operatorname{loc}}$ such that $\{\|f\psi_{\lambda}\|_E\eta(y_{\lambda})\}_{\lambda\in\Lambda}\in c_0(\Lambda)$ and we equip it with the norm \eqref{norm-amalgam-space-ultrad} with $p=\infty$.

\begin{theorem}\label{the-g-voaml}
Let $E$ be a TMIB or a DTMIB space of class $*-\dagger$ and let $\eta$ be a weight of class $\dagger$. For any UCPU $\Psi=\{\psi_{\lambda}\}_{\lambda\in\Lambda}$ of class $*-\dagger$ with points $\{y_{\lambda}\}_{\lambda\in\Lambda}$ it holds that
$$
W(E,L^p_{\eta})=W_{\Psi}(E,L^p_{\eta}),\, 1\leq p\leq\infty,\quad \mbox{and} \quad W(E,\mathcal{C}_{\eta,0})=W_{\Psi}(E,\mathcal{C}_{\eta,0})
$$
with equivalent norms.
\end{theorem}

For the proof of the theorem, we need the following technical result.

\begin{lemma}\label{lemma-est-sum-of-exponsmall}
Let $\{y_{\lambda}\}_{\lambda\in\Lambda}$ be a sequence of points in $\RR^n$ which satisfies the conditions $(2)$ and $(3)$ of Definition \ref{def-bupuultr-for-aml}. Then for every $h>0$ and $\varepsilon>0$ there exists $R>0$ such that
$$
\sup_{\mu\in\Lambda}\sum_{\{\lambda\in\Lambda\,|\, |y_{\lambda}-y_{\mu}|> R\}} e^{-A(h|y_{\lambda}-y_{\mu}|)}\leq \varepsilon.
$$
\end{lemma}

\begin{proof} Set $K:=\{x\in\RR^n\,|\, |x|\leq k\}$ where $k\in\ZZ_+$ is chosen large enough so that $\RR^n=\cup_{\lambda\in\Lambda} (y_{\lambda}+K)$. Let $C_K\geq 1$ be the constant from the condition $(2)$ in Definition \ref{def-bupuultr-for-aml} that corresponds to $K$. Set $C:=\sup_{x\in K}e^{A(h|x|)}$. Pick $R\geq k+1$ such that
$$
\int_{|x|\geq R-k}e^{-A(h|x|/2)}dx\leq \varepsilon |K|/(2CC_K).
$$
For each $\mu\in\Lambda$, denote $Q_{\mu}:=\{\lambda\in\Lambda\,|\, |y_{\lambda}-y_{\mu}|> R\}$ and estimate as follows
\begin{align*}
\sum_{\lambda\in Q_{\mu}} e^{-A(h|y_{\lambda}-y_{\mu}|)}&=\frac{1}{|K|} \sum_{\lambda\in Q_{\mu}} \int_{y_{\lambda}+K}e^{-A(h|y_{\lambda}-y_{\mu}|)}dx\leq \frac{2C}{|K|} \sum_{\lambda\in Q_{\mu}} \int_{y_{\lambda}+K}e^{-A(h|y_{\mu}-x|/2)}dx\\
&=\frac{2C}{|K|} \int_{\RR^n}e^{-A(h|y_{\mu}-x|/2)}\sum_{\lambda\in Q_{\mu}}\mathbf{1}_{y_{\lambda}+K}(x)dx.
\end{align*}
Notice that $\sum_{\lambda\in Q_{\mu}}\mathbf{1}_{y_{\lambda}+K}(x)\leq C_K$, for all $x\in\RR^n$, and $\sum_{\lambda\in Q_{\mu}}\mathbf{1}_{y_{\lambda}+K}(x)=0$ if $|x-y_{\mu}|\leq R-k$. Consequently,
\begin{align*}
\sum_{\lambda\in Q_{\mu}} e^{-A(h|y_{\lambda}-y_{\mu}|)}\leq\frac{2CC_K}{|K|} \int_{|x-y_{\mu}|\geq R-k}e^{-A(h|y_{\mu}-x|/2)}dx\leq \varepsilon.
\end{align*}
\end{proof}

\begin{proof}[Proof of Theorem \ref{the-g-voaml}] We divide the proof in three steps.\\
\indent STEP 1: The spaces $W(E,L^p_{\eta})$, $1\leq p\leq\infty$, and $W(E,\mathcal{C}_{\eta,0})$ are continuously included into $W_{\Psi}(E,L^p_{\eta})$, $1\leq p\leq \infty$, and $W_{\Psi}(E,\mathcal{C}_{\eta,0})$ respectively.\\
\indent Pick $R>0$ such that $\RR^n=\cup_{\lambda\in\Lambda}(y_{\lambda}+K)$ with $K:=\{x\in\RR^n\,|\, |x|\leq R\}$. Denote by $C_K\geq 1$ the constant from the condition $(2)$ on the UCPU $\Psi$ for the set $K$. Notice that $B:=\{T_rT_{-y_{\lambda}}\psi_{\lambda}\,|\, r\in K,\, \lambda\in\Lambda\}$ is a bounded subset of $\SSS^*_{\dagger}(\RR^n)$. For $f\in W(E,L^p_{\eta})$ and $\lambda\in\Lambda$, we have
\begin{align}
\|f\psi_{\lambda}\|_E\eta(y_{\lambda})&\leq C_1\int_{y_{\lambda}+K} \|fT_x(T_{y_{\lambda}-x}T_{-y_{\lambda}}\psi_{\lambda})\|_E\eta(x)dx\nonumber\\
&\leq C_1\int_{y_{\lambda}+K}\sup_{\chi\in B}\|fT_x\chi\|_E\eta(x)dx.\label{est-for-part-partit}
\end{align}
Consequently,
\begin{align*}
\left(\sum_{\lambda\in\Lambda}\|f\psi_{\lambda}\|_E^p\eta(y_{\lambda})^p\right)^{1/p} &\leq C_2\left(\sum_{\lambda\in\Lambda}\int_{y_{\lambda}+K}\sup_{\chi\in B}\|fT_x\chi\|_E^p\eta(x)^pdx\right)^{1/p}\\
&\leq C_2\left(\int_{\RR^n}\sup_{\chi\in B}\|fT_x\chi\|_E^p\eta(x)^p\sum_{\lambda\in\Lambda}\mathbf{1}_{y_{\lambda}+K} (x)dx\right)^{1/p}\\
&\leq C_2C_K^{1/p}\|f\|_{p,B,\eta},
\end{align*}
which proves the claim for $W_{\Psi}(E,L^p_{\eta})$, $1\leq p< \infty$, in view of Proposition \ref{lemma-amalg-space-alter-dfnts}. The claim for $W_{\Psi}(E,L^{\infty}_{\eta})$ and $W_{\Psi}(E,\mathcal{C}_{\eta,0})$ follows from \eqref{est-for-part-partit} and Proposition \ref{lemma-amalg-space-alter-dfnts}.\\
\indent STEP 2: The spaces $W(E,L^p_{\eta})$, $1\leq p\leq \infty$, and $W(E,\mathcal{C}_{\eta,0})$ are topologically imbedded into $W_{\Psi}(E,L^p_{\eta})$, $1\leq p\leq \infty$, and $W_{\Psi}(E,\mathcal{C}_{\eta,0})$ respectively.\\
\indent In view of STEP 1, it suffices to show that the topologies induced on $W(E,L^p_{\eta})$, $1\leq p\leq \infty$, and $W(E,\mathcal{C}_{\eta,0})$ by $W_{\Psi}(E,L^p_{\eta})$, $1\leq p\leq \infty$, and $W_{\Psi}(E,\mathcal{C}_{\eta,0})$ respectively are finer than their initial topologies. In the Beurling case, there are $C_1,\tau_1>0$ such that $\eta(x+y)\leq C_1\eta(x)e^{A(\tau_1|y|)}$, $x,y\in\RR^n$, and in the Roumieu case we pick any $\tau_1>0$ with the corresponding $C_1>0$ for which this inequality holds true. Let $K$ and $R,C_K\geq 1$ be as in the proof of STEP 1. Pick $\chi\in\SSS^*_{\dagger}(\RR^n)\backslash\{0\}$ and set $\chi_0:=\chi\chi\in\SSS^*_{\dagger}(\RR^n)\backslash\{0\}$. In view of Proposition \ref{lemma-amalg-space-alter-dfnts}, there exists $C'_1\geq 1$ such that $\|f\|_{p,\chi,\eta}\leq C'_1\|f\|_{p,\chi_0,\eta}$, for all $f\in W(E,L^p_{\eta})$. Lemma \ref{est-partial-growt-fourbeuralg} yields the existence of $\tilde{C}\geq 1$ and $\tilde{h}>0$ such that
$$
\sup_{x\in y_{\mu}+K}\|\psi_{\lambda}T_x\chi\|_{\mathcal{F}L^1_{\nu_E}}\leq \tilde{C}e^{-A(\tilde{h}|y_{\lambda}-y_{\mu}|)},\quad \mbox{for all}\,\, \lambda,\mu\in\Lambda
$$
(in the Beurling case, pick any $\tilde{h}>0$ with the corresponding $\tilde{C}\geq 1$). In view of Lemma \ref{lemma-est-sum-of-exponsmall}, there exists $R_1\geq 1$ such that
$$
\sup_{\mu\in\Lambda}\sum_{\{\lambda\in\Lambda\,|\,|y_{\lambda}-y_{\mu}|>R_1\}} e^{-A(\tilde{h}|y_{\lambda}-y_{\mu}|)}\leq (4\tilde{C}C'_1C_K)^{-1}.
$$
For each $\mu\in\Lambda$, define $S_{\mu}\subseteq \Lambda$ by
\begin{equation}\label{defin-of-smu}
S_{\mu}:=\{\lambda\in\Lambda\,|\, |y_{\lambda}-y_{\mu}|\leq R_1\}.
\end{equation}
Notice that there exists $k_0\in\ZZ_+$ such that $|S_{\mu}|\leq k_0$, $\forall \mu\in\Lambda$ (because of the condition (2) in Definition \ref{def-bupuultr-for-aml}). Assume first $1\leq p<\infty$ and let $f\in W(E,L^p_{\eta})$. We estimate as follows
\begin{align}\label{bound-for-f-int-sumnorm}
\|f\|_{p,\chi_0,\eta}^p\leq \int_{\RR^n}\left(\sum_{\lambda\in\Lambda}\|f\psi_{\lambda}T_x\chi_0\|_E\eta(x)\right)^pdx \leq \sum_{\mu\in\Lambda}\int_{y_{\mu}+K} \left(\sum_{\lambda\in\Lambda}\|f\psi_{\lambda}T_x\chi_0\|_E\eta(x)\right)^pdx.
\end{align}
For $\mu\in\Lambda$ and $x\in y_{\mu}+K$ we have
\begin{multline*}
\left(\sum_{\lambda\in\Lambda}\|f\psi_{\lambda}T_x\chi_0\|_E\eta(x)\right)^p\\
\leq 2^p \left(\sum_{\lambda\in S_{\mu}}\|f\psi_{\lambda}T_x\chi_0\|_E\eta(x)\right)^p +2^p \left(\sum_{\lambda\in \Lambda\backslash S_{\mu}} \|f\psi_{\lambda}T_x\chi_0\|_E\eta(x)\right)^p\leq I_1+I_2,
\end{multline*}
with
$$
I_1:=(2k_0)^p\|\chi_0\|_{\mathcal{F}L^1_{\nu_E}}^p \sum_{\lambda\in S_{\mu}}\|f\psi_{\lambda}\|_E^p\eta(x)^p,\quad I_2:=2^p \left(\sum_{\lambda\in \Lambda\backslash S_{\mu}}\|f\psi_{\lambda}T_x\chi_0\|_E\eta(x)\right)^p.
$$
We estimate $I_1$ as follows
$$
I_1\leq (2k_0C_1)^pe^{pA(\tau_1R)}\|\chi_0\|_{\mathcal{F}L^1_{\nu_E}}^p \sum_{\lambda\in S_{\mu}}\|f\psi_{\lambda}\|_E^p\eta(y_{\mu})^p\\
\leq C_2\sum_{\lambda\in S_{\mu}}\|f\psi_{\lambda}\|_E^p\eta(y_{\lambda})^p,
$$
with $C_2:=(2k_0C_1^2)^pe^{pA(\tau_1R)}e^{pA(\tau_1R_1)}\|\chi_0\|_{\mathcal{F}L^1_{\nu_E}}^p$ (notice that $C_2$ does not depend on $f$). To estimate $I_2$, we proceed as follows
\begin{align*}
I_2&\leq 2^p \|fT_x\chi\|_E^p\eta(x)^p\left(\sum_{\lambda\in\Lambda\backslash S_{\mu}}\|\psi_{\lambda}T_x\chi\|_{\mathcal{F}L^1_{\nu_E}}\right)^p\\
&\leq 2^p\tilde{C}^p \|fT_x\chi\|_E^p\eta(x)^p \left(\sum_{\lambda\in\Lambda\backslash S_{\mu}} e^{-A(\tilde{h}|y_{\lambda}-y_{\mu}|)}\right)^p\leq (2C'_1C_K)^{-p}\|fT_x\chi\|_E^p\eta(x)^p.
\end{align*}
Plugging these bounds in \eqref{bound-for-f-int-sumnorm}, we infer
\begin{align*}
\|f\|_{p,\chi_0,\eta}^p&\leq C_2|K| \sum_{\mu\in\Lambda}\sum_{\lambda\in S_{\mu}}\|f\psi_{\lambda}\|_E^p\eta(y_{\lambda})^p+(2C'_1C_K)^{-p} \sum_{\mu\in\Lambda}\int_{y_{\mu}+K} \|fT_x\chi\|_E^p\eta(x)^pdx\\
&\leq C_2|K| \sum_{\lambda\in\Lambda}\|f\psi_{\lambda}\|_E^p\eta(y_{\lambda})^p \sum_{\mu\in\Lambda}\mathbf{1}_{S_{\mu}}(\lambda)\\
&{}\quad+(2C'_1C_K)^{-p}\int_{\RR^n} \|fT_x\chi\|_E^p\eta(x)^p\sum_{\mu\in\Lambda}\mathbf{1}_{y_{\mu}+K}(x)dx\\
&\leq C_2C_3|K| \sum_{\lambda\in\Lambda}\|f\psi_{\lambda}\|_E^p\eta(y_{\lambda})^p +2^{-p}\|f\|_{p,\chi_0,\eta}^p,
\end{align*}
where we denoted by $C_3$ the constant from the condition $(2)$ on the UCPU $\Psi$ for the compact set $\{x\in\RR^n\,|\, |x|\leq R_1\}$. This gives $\|f\|_{p,\chi_0,\eta}\leq 2(C_2C_3|K|)^{1/p}\|f\|_{p,\Psi,\eta}$. Since the constants $C_2$, $C_3$ and $R$ do not depend on $f$, this proves the claim when $1\leq p<\infty$. The proof for the case $p=\infty$ is analogous and we omit it. The latter immediately implies the claim in STEP 2 for the space $W(E,\mathcal{C}_{\eta,0})$ since $W(E,\mathcal{C}_{\eta,0})$ and $W_{\Psi}(E,\mathcal{C}_{\eta,0})$ are topologically imbedded into $W(E,L^{\infty}_{\eta})$ and $W_{\Psi}(E,L^{\infty}_{\eta})$ respectively.\\
\indent STEP 3: The spaces $W(E,L^p_{\eta})$, $1\leq p\leq \infty$, and $W(E,\mathcal{C}_{\eta,0})$ are topologically isomorphic to $W_{\Psi}(E,L^p_{\eta})$, $1\leq p\leq \infty$, and $W_{\Psi}(E,\mathcal{C}_{\eta,0})$ respectively.\\
\indent In view of STEP 2, it suffices to show that $W_{\Psi}(E,L^p_{\eta})\subseteq W(E,L^p_{\eta})$, $1\leq p\leq \infty$, and $W_{\Psi}(E,\mathcal{C}_{\eta,0})\subseteq W(E,\mathcal{C}_{\eta,0})$. Because of Proposition \ref{facto-s-mul}, for each $\chi\in\SSS^*_{\dagger}(\RR^n)$, the mappings $E_{\operatorname{loc}}\rightarrow W(E,L^p_{\eta})$, $f\mapsto f\chi$, and $E_{\operatorname{loc}}\rightarrow W(E,\mathcal{C}_{\eta,0})$, $f\mapsto f\chi$, are well-defined and continuous. Pick $\phi\in\SSS^{\dagger}_*(\RR^n)$ such that $\int_{\RR^n}\phi(x)dx=1$ and set $\phi_j(x):=j^n\phi(jx)$, $x\in\RR^n$, $j\in\ZZ_+$. Define $\varphi_j:=\mathcal{F}\phi_j\in\SSS^*_{\dagger}(\RR^n)\backslash\{0\}$, $j\in\ZZ_+$; clearly, $\sup_{j\in\ZZ_+}\|\varphi_j\|_{\mathcal{F}L^1_{\nu_E}}<\infty$. Let $f\in W_{\Psi}(E,L^p_{\eta})$. Fix $\chi\in\SSS^*_{\dagger}(\RR^n)\backslash\{0\}$. In view of STEP 2, there exists $C\geq 1$ such that $\|f\varphi_j\|_{p,\chi,\eta}\leq C\|f\varphi_j\|_{p,\Psi,\eta}$, for all $j\in\ZZ_+$. For each fixed $x\in\RR^n$, $\varphi_jT_x\chi\rightarrow T_x\chi$ in $\SSS^*_{\dagger}(\RR^n)$. Hence, when $p\in[1,\infty)$, Lemma \ref{lemma-reg-e-loc-proper} $(ii)$ together with Fatou's lemma imply
$$
\|f\|_{p,\chi,\eta}^p\leq \liminf_{j\rightarrow \infty}\|f\varphi_j\|_{p,\chi,\eta}^p\leq C^p(\sup_{j\in\ZZ_+}\|\varphi_j\|_{\mathcal{F}L^1_{\nu_E}})^p\|f\|_{p,\Psi,\eta}^p,
$$
and hence $f\in W(E,L^p_{\eta})$. The proof in the case $p=\infty$ is analogous and we omit it. It remains to prove the claim for $W_{\Psi}(E,\mathcal{C}_{\eta,0})$. We may assume that $\eta$ is continuous. Let $\chi$, $\chi_0$, $K$, $R$, $\tau_1$, $C_1$, $\tilde{C}$ and $\tilde{h}$ be as in the proof of STEP 2. Let $f\in W_{\Psi}(E,\mathcal{C}_{\eta,0})\backslash\{0\}$; the above implies that $\|f\|_{\infty,\chi,\eta}<\infty$. Let $\varepsilon>0$ be arbitrary but fixed. Lemma \ref{lemma-est-sum-of-exponsmall} implies that there exists $R_1\geq 1$ such that
$$
\sup_{\mu\in\Lambda}\sum_{\{\lambda\in\Lambda\,|\,|y_{\lambda}-y_{\mu}|>R_1\}} e^{-A(\tilde{h}|y_{\lambda}-y_{\mu}|)}\leq \varepsilon/(2\tilde{C}\|f\|_{\infty,\chi,\eta}).
$$
For $\mu\in\Lambda$, set $S_{\mu}:=\{\lambda\in\Lambda\,|\, |y_{\lambda}-y_{\mu}|\leq R_1\}$. There is $k_0\in\ZZ_+$ such that $|S_{\mu}|\leq k_0$, for all $\mu\in\Lambda$. There exists a finite $\Lambda_0\subseteq \Lambda$ such that
$$
\|f\psi_{\lambda}\|_E\eta(y_{\lambda})\leq (2k_0C_1^2e^{A(\tau_1R)}e^{A(\tau_1 R_1)} \|\chi_0\|_{\mathcal{F}L^1_{\nu_E}})^{-1}\varepsilon,\quad \mbox{for all}\,\, \lambda\in\Lambda\backslash \Lambda_0.
$$
Pick $R_0\geq 1$ such that $|y_{\lambda}|\leq R_0$, for all $\lambda\in\Lambda_0$. Let $x\in\RR^n$ be such that $|x|\geq R_0+R+R_1+1$. There exists $\mu\in\Lambda$ such that $x\in y_{\mu}+K$. We have
\begin{align}
\|fT_x\chi_0\|_E\eta(x)&\leq \sum_{\lambda\in\Lambda}\|f\psi_{\lambda}T_x\chi_0\|_E\eta(x)\nonumber\\
&\leq \|\chi_0\|_{\mathcal{F}L^1_{\nu_E}}\sum_{\lambda\in S_{\mu}}\|f\psi_{\lambda}\|_E\eta(x) +\|f\|_{\infty,\chi,\eta}\sum_{\lambda\in\Lambda\backslash S_{\mu}}\|\psi_{\lambda}T_x\chi\|_{\mathcal{F}L^1_{\nu_E}}\nonumber\\
&\leq C_1e^{A(\tau_1R)}\|\chi_0\|_{\mathcal{F}L^1_{\nu_E}}\sum_{\lambda\in S_{\mu}}\|f\psi_{\lambda}\|_E\eta(y_{\mu}) +\varepsilon/2\nonumber\\
&\leq C_1^2e^{A(\tau_1R)}e^{A(\tau_1R_1)} \|\chi_0\|_{\mathcal{F}L^1_{\nu_E}}\sum_{\lambda\in S_{\mu}}\|f\psi_{\lambda}\|_E\eta(y_{\lambda}) +\varepsilon/2.\label{est-for-nor-for-far-x}
\end{align}
We claim that $S_{\mu}\subseteq \Lambda\backslash\Lambda_0$. To verify this, assume that there is $\lambda\in S_{\mu}$ such that $\lambda\in\Lambda_0$. Then
$$
R_0+R+R_1+1\leq |x|\leq |x-y_{\mu}|+|y_{\mu}-y_{\lambda}|+|y_{\lambda}|\leq R+R_1+R_0,
$$
which is a contradiction. Now, \eqref{est-for-nor-for-far-x} implies $\|fT_x\chi_0\|_E\eta(x)\leq \varepsilon$ which completes the proof of STEP 3, which, in turn, verifies the claim in the theorem.
\end{proof}

\begin{remark}
By employing the same technique, one can show Theorem \ref{the-g-voaml} when $E$ is a (D)TMIB space of distributions, $\eta$ a polynomially bounded weight and $\Psi=\{\psi_{\lambda}\}_{\lambda\in\Lambda}$ a UCPU for $\SSS(\RR^n)$ as in Remark \ref{bupufor-s-dis} (cf. Remark \ref{rem-amalgam-for-dis-withs-fun}).
\end{remark}

\section{Duals of amalgam spaces and complex interpolation}\label{dual amalgam}

\subsection{Duals of amalgam spaces}
In this subsection we identify the duals of most of the amalgam spaces we defined above. We start by showing that, when $E$ is a TMIB and $p<\infty$, the amalgam spaces are also TMIB spaces.

\begin{proposition}\label{tmib}
Let $E$ be TMIB or a DTMIB space of class $*-\dagger$, $\eta$ a weight of class $\dagger$ and $1\leq p_1\leq p_2<\infty$. Then the following continuous inclusions hold true:
$$
\SSS^*_{\dagger}(\RR^n)\rightarrow W(E, L^{p_1}_{\eta})\rightarrow W(E,L^{p_2}_{\eta}) \rightarrow W(E,\mathcal{C}_{\eta,0})\rightarrow W(E,L^{\infty}_{\eta})\rightarrow E_{\operatorname{loc}}\rightarrow \SSS'^*_{\dagger}(\RR^n).
$$
Furthermore, if $E$ is a TMIB space of class $*-\dagger$ then $W(E,L^p_{\eta})$, $1\leq p<\infty$, and $W(E,\mathcal{C}_{\eta,0})$ are also TMIB spaces of class $*-\dagger$.
\end{proposition}

\begin{proof} It is straightforward to verify that $\SSS^*_{\dagger}(\RR^n)\subseteq W(E,L^1_{\eta})$ and $W(E,L^{\infty}_{\eta})\subseteq E_{\operatorname{loc}}$ continuously. The rest of the first part follows from Lemma \ref{lemma-locspa-continc}, Theorem \ref{the-g-voaml} and the continuous inclusions $\ell^{p_1}(\Lambda)\subseteq \ell^{p_2}(\Lambda)\subseteq c_0(\Lambda)\subseteq \ell^{\infty}(\Lambda)$, $1\leq p_1\leq p_2<\infty$.\\
\indent Let $E$ be a TMIB. The only non-trivial part of the second claim in the proposition is the density of $\SSS^*_{\dagger}(\RR^n)$ in $W(E,L^p_{\eta})$, $1\leq p<\infty$, and in $W(E,\mathcal{C}_{\eta,0})$. Without loss of generality, we can assume that $\eta$ is continuous. Pick $\phi\in\SSS^{\dagger}_*(\RR^n)$ and $\psi\in\SSS^*_{\dagger}(\RR^n)$ such that $\int_{\RR^n}\phi(x)dx=1$ and $\int_{\RR^n}\psi(x)dx=1$. Set $\psi_j(x):=j^n\psi(jx)$ and $\phi_j(x):=j^n\phi(jx)$, $x\in\RR^n$, $j\in\ZZ_+$. Define $\varphi_j:=\mathcal{F}\phi_j\in\SSS^*_{\dagger}(\RR^n)\backslash\{0\}$, $j\in\ZZ_+$. Then, for each $e\in E$, $e*\psi_j\rightarrow e$ and $e\varphi_j\rightarrow e$ in $E$ (this easily follows from \eqref{integ-form-for-con-mult}). Clearly, $C':=\sup_{j\in\ZZ_+}\|\varphi_j\|_{\mathcal{F}L^1_{\nu_E}}<\infty$. In the Beurling case, there exist $C,\tau>0$ such that
$$
\omega_E(x)\leq Ce^{A(\tau|x|)},\,\, x\in\RR^n,\quad \mbox{and}\quad \eta(x+y)\leq C\eta(x)e^{A(\tau|y|)},\,\, x,y\in\RR^n.
$$
In the Roumieu case, pick $\tau>0$ such that $e^{3A(\tau|\cdot|)}\psi \in L^1(\RR^n)$ and the corresponding $C>0$ in the above inequalities. Let $f\in W(E, L^p_{\eta})$ and $\chi\in\SSS^*_{\dagger}(\RR^n)\backslash\{0\}$. Let $\varepsilon>0$ be arbitrary but fixed. Since
\begin{equation}\label{ine-for-est-domco}
\|(f\varphi_j-f)T_x\chi\|^p_E\leq 2^p\|f\varphi_jT_x\chi\|^p_E+2^p\|fT_x\chi\|^p_E\leq 2^p(C'^p+1)\|fT_x\chi\|^p_E,
\end{equation}
dominated convergence implies that there is $j_0\in\ZZ_+$ such that $\|f\varphi_{j_0}-f\|_{p,\chi,\eta}\leq \varepsilon/2$. In view of \eqref{integ-form-for-con-mult}, we have
\begin{equation}\label{est-for-dif-with-convdens}
\|((f\varphi_{j_0})*\psi_k-f\varphi_{j_0})T_x\chi\|_E
\leq \int_{\RR^n}\|(T_y(f\varphi_{j_0})-f\varphi_{j_0})T_x\chi\|_E |\psi_k(y)|dy.
\end{equation}
The Minkowski integral inequality gives
\begin{multline}\label{ine-for-minbou-for-con}
\|(f\varphi_{j_0})*\psi_k-f\varphi_{j_0}\|_{p,\chi,\eta}\\
\leq \int_{\RR^n}\left(\int_{\RR^n}\|(T_y(f\varphi_{j_0})-f\varphi_{j_0})T_x\chi\|^p_E \eta(x)^pdx\right)^{1/p}|\psi_k(y)|dy.
\end{multline}
We claim that
$$
\int_{\RR^n}\|(T_y(f\varphi_{j_0})-f\varphi_{j_0})T_x\chi\|^p_E \eta(x)^pdx\rightarrow 0,\quad \mbox{as}\,\, y\rightarrow 0.
$$
This follows from dominated convergence since for $y$ in a compact set $K\subseteq \RR^n$ we have
\begin{align}
\|T_y(f\varphi_{j_0})T_x\chi\|^p_E\eta(x)^p&\leq \|f\varphi_{j_0}T_xT_{-y}\chi\|^p_E\omega_E(y)^p\eta(x)^p\nonumber\\
&\leq C_1\eta(x)^p \sup_{y\in K}\|fT_xT_{-y}\chi\|^p_E\label{edst-for-densconvsec}
\end{align}
and the function $x\mapsto \eta(x)^p\sup_{y\in K}\|fT_xT_{-y}\chi\|^p_E$ belongs to $L^1(\RR^n)$ in view of Proposition \ref{lemma-amalg-space-alter-dfnts}. Hence, there exists $\varepsilon_1>0$ such that
$$
\int_{\RR^n}\|(T_y(f\varphi_{j_0})-f\varphi_{j_0})T_x\chi\|^p_E \eta(x)^pdx\leq \frac{\varepsilon^p}{4^p\|\psi\|^p_{L^1(\RR^n)}},\quad \mbox{when}\,\, |y|\leq \varepsilon_1.
$$
Pick $k_0\in\ZZ_+$ such that
\begin{equation}\label{ineq-for-great-bou-dens}
\int_{|y|\geq k_0\varepsilon_1}e^{2A(\tau|y|)}|\psi(y)|dy\leq (4(1+C^2)(1+\|f\varphi_{j_0}\|_{p,\chi,\eta}))^{-1} \varepsilon.
\end{equation}
Then, \eqref{ine-for-minbou-for-con} gives
\begin{align*}
\|&(f\varphi_{j_0})*\psi_{k_0}-f\varphi_{j_0}\|_{p,\chi,\eta}\\
&\leq \frac{\varepsilon}{4} +\int_{|y|\geq\varepsilon_1} \left(\|f\varphi_{j_0}\|_{p,\chi,\eta}+ \omega_E(y) \left(\int_{\RR^n}\|f\varphi_{j_0}T_{x-y}\chi\|^p_E\eta(x)^pdx\right)^{1/p}\right) |\psi_{k_0}(y)|dy\\
&\leq \frac{\varepsilon}{4} +\int_{|y|\geq k_0\varepsilon_1}(\|f\varphi_{j_0}\|_{p,\chi,\eta}+ C^2e^{2A(\tau|y|/k_0)} \|f\varphi_{j_0}\|_{p,\chi,\eta}) |\psi(y)|dy\leq \frac{\varepsilon}{2}.
\end{align*}
Consequently $\|(f\varphi_{j_0})*\psi_{k_0}-f\|_{p,\chi,\eta}\leq \varepsilon$ and the required density follows since $(f\varphi_{j_0})*\psi_{k_0}\in\SSS^*_{\dagger}(\RR^n)$. Assume now that $f\in W(E, \mathcal{C}_{\eta,0})$. Let $C,\tau>0$ be as above and let $\varepsilon>0$ be arbitrary but fixed. In view of \eqref{ine-for-est-domco}, there is a compact subset $K$ of $\RR^n$ such that $\sup_{j\in\ZZ_+}\|(f\varphi_j-f)T_x\chi\|_E\eta(x)\leq \varepsilon/2$, $x\in\RR^n\backslash K$. Let $S_j\in\mathcal{L}(E)$, $j\in\ZZ_+$, be the operator $S_je=e\varphi_j$. Then $S_j\rightarrow \operatorname{Id}$ in the topology of simple convergence and the Banach-Steinhaus theorem implies that the convergence holds in the topology of precompact convergence. Lemma \ref{lemma-reg-e-loc-proper} $(iii)$ implies that $\{fT_x\chi\, |\, x\in K\}$ is a compact subset of $E$ and consequently there is $j_0\in \ZZ_+$ such that $\|(f\varphi_{j_0}-f)T_x\chi\|_E\leq \varepsilon/(2\|\eta\|_{L^{\infty}(K)})$, $x\in K$. Hence, $\|f\varphi_{j_0}-f\|_{\infty,\chi,\eta}\leq \varepsilon/2$. The inequality \eqref{est-for-dif-with-convdens} holds true. In view of \eqref{edst-for-densconvsec} with $p=1$ and Proposition \ref{lemma-amalg-space-alter-dfnts}, there is a compact set $K_1\subseteq \RR^n$ such that
$$
\|(T_y(f\varphi_{j_0})-f\varphi_{j_0})T_x\chi\|_E\eta(x)\leq \varepsilon/(4\|\psi\|_{L^1(\RR^n)}),\quad x\in \RR^n\backslash K_1,\, |y|\leq 1.
$$
Since $\RR^n\rightarrow E$, $y\mapsto T_y(f\varphi_{j_0})$, is continuous, there exists $0<\varepsilon_1\leq 1$ such that
$$
\|T_y(f\varphi_{j_0})-f\varphi_{j_0}\|_E\leq \varepsilon/(4\|\eta\|_{L^{\infty}(K_1)}\|\chi\|_{\mathcal{F}L^1_{\nu_E}} \|\psi\|_{L^1(\RR^n)}),\quad \mbox{when}\,\, |y|\leq \varepsilon_1.
$$
Pick $k_0\in\ZZ_+$ such that \eqref{ineq-for-great-bou-dens} holds true with $p=\infty$. Then, for $x\in K_1$, \eqref{est-for-dif-with-convdens} gives
\begin{align*}
\|&((f\varphi_{j_0})*\psi_{k_0}-f\varphi_{j_0})T_x\chi\|_E\eta(x)\\
&\leq \|\chi\|_{\mathcal{F}L^1_{\nu_E}}\|\eta\|_{L^{\infty}(K_1)} \int_{|y|\leq \varepsilon_1} \|T_y(f\varphi_{j_0})-f\varphi_{j_0}\|_E|\psi_{k_0}(y)|dy\\
&{}\quad+\int_{|y|\geq\varepsilon_1} \left(\|f\varphi_{j_0}\|_{\infty,\chi,\eta}+ \omega_E(y) \|f\varphi_{j_0}T_{x-y}\chi\|_E\eta(x)\right) |\psi_{k_0}(y)|dy\\
&\leq \frac{\varepsilon}{4} +\int_{|y|\geq k_0\varepsilon_1} \left(\|f\varphi_{j_0}\|_{\infty,\chi,\eta}+ C^2e^{2A(\tau|y|/k_0)} \|f\varphi_{j_0}\|_{\infty,\chi,\eta}\right) |\psi(y)|dy\leq \frac{\varepsilon}{2}.
\end{align*}
Similarly, $\|((f\varphi_{j_0})*\psi_{k_0}-f\varphi_{j_0})T_x\chi\|_E \eta(x)\leq \varepsilon/2$ when $x\in\RR^n\backslash K_1$. We conclude $\|(f\varphi_{j_0})*\psi_{k_0}-f\|_{\infty,\chi,\eta}\leq \varepsilon$, which proves the density of $\SSS^*_{\dagger}(\RR^n)$ in $W(E,\mathcal{C}_{\eta,0})$.
\end{proof}

Because of this proposition, when $E$ is a TMIB space of class $*-\dagger$, the duals of $W(E,L^p_{\eta})$, $1\leq p<\infty$, and $W(E,\mathcal{C}_{\eta,0})$ are subspaces of $\SSS'^*_{\dagger}(\RR^n)$. The following result identifies these spaces.

\begin{theorem}\label{duali}
Let $\eta$ be a weight of class $\dagger$ and let $E$ be a TMIB space of class $*-\dagger$. The strong dual $(W(E,L^p_\eta))'_b$, $1\leq p<\infty$, of $W(E,L^p_\eta)$ is equal to $W(E',L^q_{1/\eta})$, $p^{-1}+q^{-1}=1$, with equivalent norms. The strong dual $(W(E,\mathcal{C}_{\eta,0}))'_b$ of $W(E,\mathcal{C}_{\eta,0})$ is equal to $W(E', L^1_{1/\eta})$ with equivalent norms. Consequently, $W(E',L^p_{1/\eta})$, $1\leq p\leq \infty$, is a DTMIB space of class $*-\dagger$.
\end{theorem}

\begin{proof} Proposition \ref{tmib} implies that $(W(E,L^p_{\eta}))'_b$, $1\leq p<\infty$, and $(W(E,\mathcal{C}_{\eta,0}))'_b$ are continuously included into $\SSS'^*_{\dagger}(\RR^n)$. Denote by $1< q\leq \infty$ the H\"older conjugate index to $p$ and let $q=1$ in the case of $(W(E,\mathcal{C}_{\eta,0}))'_b$.\\
\indent Let $f\in W(E',L^q_{1/\eta})$. Pick $\chi_0\in\SSS^*_{\dagger}(\RR^n)\backslash\{0\}$ and set $\chi:=\|\chi_0\|^{-2}_{L^2(\RR^n)}\chi_0\overline{\chi_0}\in \SSS^*_{\dagger}(\RR^n)\backslash\{0\}$; clearly $\int_{\RR^n}\chi(y)dy=1$. Let $\varphi\in\SSS^*_{\dagger}(\RR^n)$ be arbitrary but fixed. The function $\RR^{2n}\rightarrow \CC$, $(x,y)\mapsto \chi(y-x)\varphi(y)$, belongs to $\SSS^*_{\dagger}(\RR^{2n})$. We infer
\begin{align*}
|\langle f,\varphi\rangle|&=|\langle \mathbf{1}_{\RR^n}(x)\otimes f(y), \chi(y-x)\varphi(y)\rangle|=\|\chi_0\|^{-2}_{L^2(\RR^n)}\left|\int_{\RR^n}\langle f T_x\chi_0,\varphi T_x\overline{\chi_0}\rangle dx\right|\\
&\leq\|\chi_0\|^{-2}_{L^2(\RR^n)}\int_{\RR^n}\|f T_x\chi_0\|_{E'}\|\varphi T_x\overline{\chi_0}\|_E dx\leq \|\chi_0\|^{-2}_{L^2(\RR^n)} \|f\|_{q,\chi_0,1/\eta}\|\varphi\|_{p,\overline{\chi_0},\eta};
\end{align*}
of course, $p=\infty$ when $q=1$. In view of Proposition \ref{tmib}, this verifies that $W(E',L^q_{1/\eta})\subseteq (W(E,L^p_{\eta}))'_b$, $1<q\leq \infty$, and $W(E',L^1_{1/\eta})\subseteq (W(E,\mathcal{C}_{\eta,0}))'_b$ continuously.\\
\indent To prove the opposite inclusion, first we show that $(W(E,L^p_{\eta}))'\subseteq E'_{\operatorname{loc}}$, $1\leq p<\infty$, and $(W(E,\mathcal{C}_{\eta,0}))'\subseteq E'_{\operatorname{loc}}$. Let $f\in (W(E,L^p_{\eta}))'$. Let $\chi\in\SSS^*_{\dagger}(\RR^n)\backslash\{0\}$ and take $\|\cdot\|_{p,\chi,\eta}$ as a norm on $W(E,L^p_{\eta})$. For $\varphi,\psi\in\SSS^*_{\dagger}(\RR^d)$, we estimate as follows
\begin{align*}
|\langle \varphi f,\psi\rangle|\leq \|f\|_{(W(E,L^p_{\eta}))'_b} \|\varphi\psi\|_{p,\chi,\eta}\leq \|f\|_{(W(E,L^p_{\eta}))'_b} \|\psi\|_E\left(\int_{\RR^n}\|\varphi T_x\chi\|^p_{\mathcal{F}L^1_{\nu_E}}\eta(x)^pdx\right)^{1/p}.
\end{align*}
The very last integral is finite since $\SSS^*_{\dagger}(\RR^n)\subseteq W(\mathcal{F}L^1_{\nu_E}, L^p_{\eta})$. Hence $\varphi f\in E'$. The proof of $(W(E,\mathcal{C}_{\eta,0}))'_b\subseteq E'_{\operatorname{loc}}$ is analogous and we omit it. When $p=1$, taking $\varphi=T_y\chi$ in the above inequality, we deduce
\begin{align*}
\|fT_y\chi\|_{E'}&\leq \|f\|_{(W(E,L^1_{\eta}))'_b}\int_{\RR^n}\|\chi T_{x-y}\chi\|_{\mathcal{F}L^1_{\nu_E}}\eta(x)dx\\
&\leq C\|f\|_{(W(E,L^1_{\eta}))'_b}\eta(y)\int_{\RR^n}\|\chi T_x\chi\|_{\mathcal{F}L^1_{\nu_E}}e^{A(\tau|x|)}dx,
\end{align*}
with the specific $\tau>0$ coming from the weight $\eta$ in the Beurling case and, in the Roumieu case, we choose $\tau>0$ so that the integral is convergent. This immediately implies that $(W(E,L^1_{\eta}))'_b\subseteq W(E',L^{\infty}_{1/\eta})$ continuously.\\
\indent Assume now $1<p<\infty$. Let $\Psi=\{\psi_{\lambda}\}_{\lambda\in\Lambda}$ be a UCPU of class $*-\dagger$ with points $\{y_{\lambda}\}_{\lambda\in\Lambda}$ and set $C_0:=\sup_{\lambda\in\Lambda}\|\psi_{\lambda}\|_{\mathcal{F}L^1_{\nu_E}}<\infty$ (cf. Remark \ref{boun-bupuforlealge-nor}). Take $\|\cdot\|_{p,\Psi,\eta}$ as a norm on $W(E, L^p_{\eta})$. In the Beurling case, there are $C_1,\tau>0$ such that
\begin{equation}\label{est-forweight-s}
\eta(x+y)\leq C_1\eta(x)e^{A(\tau|y|)},\quad x,y\in\RR^n.
\end{equation}
Set $h:=H\tau$. Hence
\begin{equation}\label{chose-ofh-forconvsss}
e^{A(\tau|x|)}e^{-A(h|x|)}\leq c_0e^{-A(h|x|/H)},\quad x\in\RR^n.
\end{equation}
In view of Lemma \ref{cor-for-dis-betw-two-elebupu}, there is $C>0$ such that
\begin{equation}\label{est-formixpro-cass}
\|\psi_{\lambda}\psi_{\mu}\|_{\mathcal{F}L^1_{\nu_E}}\leq Ce^{-A(h|y_{\lambda}-y_{\mu}|)},\quad \lambda,\mu\in\Lambda.
\end{equation}
In the Roumieu case, Lemma \ref{cor-for-dis-betw-two-elebupu} implies that there are $C,h>0$ such that \eqref{est-formixpro-cass} holds true. Set $\tau:=h/H$ and pick $C_1>0$ so that the bound \eqref{est-forweight-s} is valid. Then the inequality \eqref{chose-ofh-forconvsss} holds true. In view of Lemma \ref{lemma-for-familyof-points-b},
$$
C_2:=\sup_{\mu\in\Lambda}\sum_{\lambda\in\Lambda} e^{-A(h|y_{\lambda}-y_{\mu}|/H^2)}<\infty\quad \mbox{and}\quad C_3:=\sum_{\lambda\in\Lambda} (1+|y_{\lambda}|)^{-n-1}<\infty.
$$
Let $f\in (W(E, L^p_{\eta}))'_b\backslash\{0\}$ and let $\varepsilon\in (0,1)$ be arbitrary but fixed. Set $\tilde{f}:=\|f\|_{(W(E,L^p_{\eta}))'_b}^{-1}f$. For every $\lambda\in\Lambda$, there exists $\chi_{\lambda}\in\SSS^*_{\dagger}(\RR^d)$ satisfying $\|\chi_{\lambda}\|_E\leq 1$, such that $\langle \psi_{\lambda}\tilde{f},\chi_{\lambda}\rangle\geq 0$ and
\begin{equation}\label{est-f-ch-funcc1}
\|\psi_{\lambda}\tilde{f}\|_{E'}\leq \langle \psi_{\lambda}\tilde{f},\chi_{\lambda}\rangle+C_3^{-1}(1+|y_{\lambda}|)^{-n-1} \eta(y_{\lambda})\varepsilon.
\end{equation}
For each finite $\Lambda'\subseteq \Lambda$, define
\begin{equation}\label{fun-def-for-the-est1}
\varphi_{\Lambda'}:=\sum_{\lambda\in\Lambda'}\langle \psi_{\lambda}\tilde{f},\chi_{\lambda}\rangle^{q/p} \eta(y_{\lambda})^{-q} \psi_{\lambda}\chi_{\lambda}\in\SSS^*_{\dagger}(\RR^n).
\end{equation}
Employing $e^{-A(h|x|/H)}\leq c_0e^{-2A(h|x|/H^2)}$, $x\in\RR^n$, and the H\"older inequality, we infer
\begin{align}
\|\varphi_{\Lambda'}\|_{p,\Psi,\eta}^p&\leq\sum_{\mu\in\Lambda}\left(\sum_{\lambda\in \Lambda'} \langle \psi_{\lambda}\tilde{f},\chi_{\lambda}\rangle^{q/p} \eta(y_{\lambda})^{-q}\eta(y_{\mu}) \|\psi_{\lambda}\psi_{\mu}\chi_{\lambda}\|_E\right)^p\nonumber\\
&\leq C^pC_1^p\sum_{\mu\in\Lambda}\left(\sum_{\lambda\in \Lambda'} \langle \psi_{\lambda}\tilde{f},\chi_{\lambda}\rangle^{q/p} \eta(y_{\lambda})^{-q+1}e^{A(\tau|y_{\lambda}-y_{\mu}|)} e^{-A(h|y_{\lambda}-y_{\mu}|)}\right)^p\nonumber\\
&\leq c_0^{2p}C^pC_1^p\sum_{\mu\in\Lambda}\left(\sum_{\lambda\in \Lambda'} \langle \psi_{\lambda}\tilde{f},\chi_{\lambda}\rangle^{q/p} \eta(y_{\lambda})^{-q/p} e^{-2A(h|y_{\lambda}-y_{\mu}|/H^2)}\right)^p\nonumber\\
&\leq  c_0^{2p}C^pC_1^p\sum_{\mu\in\Lambda}\left(\sum_{\lambda\in \Lambda'} \frac{\langle \psi_{\lambda}\tilde{f},\chi_{\lambda}\rangle^q}{\eta(y_{\lambda})^q} e^{-pA(h|y_{\lambda}-y_{\mu}|/H^2)}\right)\left(\sum_{\lambda\in \Lambda'} e^{-qA(h|y_{\lambda}-y_{\mu}|/H^2)}\right)^{p/q}\nonumber\\
&\leq c_0^{2p}C^pC_1^pC_2^p\sum_{\lambda\in \Lambda'} \langle \psi_{\lambda}\tilde{f},\chi_{\lambda}\rangle^q \eta(y_{\lambda})^{-q} \sum_{\mu\in\Lambda}e^{-pA(h|y_{\lambda}-y_{\mu}|/H^2)}\nonumber\\
&\leq c_0^{2p}C^pC_1^pC_2^{p+1}\sum_{\lambda\in \Lambda'} \langle \psi_{\lambda}\tilde{f},\chi_{\lambda}\rangle^q \eta(y_{\lambda})^{-q}.\label{est-for-the-pa-inthno1}
\end{align}
Notice that $\langle \tilde{f},\varphi_{\Lambda'}\rangle=\sum_{\lambda\in\Lambda'}\langle \psi_{\lambda}\tilde{f},\chi_{\lambda}\rangle^q\eta(y_{\lambda})^{-q}$. Since $\langle \tilde{f},\varphi_{\Lambda'}\rangle\leq \|\varphi_{\Lambda'}\|_{p,\Psi,\eta}$ and since $\Lambda'$ is arbitrary, \eqref{est-for-the-pa-inthno1} gives
\begin{align*}
\left(\sum_{\lambda\in\Lambda}\langle\psi_{\lambda}\tilde{f}, \chi_{\lambda}\rangle^q\eta(y_{\lambda})^{-q}\right)^{1/q}\leq \tilde{C},
\end{align*}
with $\tilde{C}:=c_0^2CC_1C_2^{1+1/p}$. In view of \eqref{est-f-ch-funcc1}, we deduce
\begin{align*}
\left(\sum_{\lambda\in\Lambda}\frac{\|\psi_{\lambda}\tilde{f}\|_{E'}^q} {\eta(y_{\lambda})^q}\right)^{1/q}\leq \left(\sum_{\lambda\in\Lambda}\frac{\langle\psi_{\lambda}\tilde{f}, \chi_{\lambda}\rangle^q}{\eta(y_{\lambda})^q}\right)^{1/q}+\varepsilon\leq \tilde{C}+\varepsilon.
\end{align*}
This implies $f\in W(E',L^q_{1/\eta})$ and, as $\varepsilon>0$ is arbitrary, we conclude $\|\{\|\psi_{\lambda}f\|_{E'} \eta(y_{\lambda})^{-1}\}_{\lambda\in\Lambda}\|_{\ell^q(\Lambda)}\leq \tilde{C}\|f\|_{(W(E,L^p_{\eta}))'_b}$; this completes the proof when $1<p<\infty$. The case when $f\in (W(E,\mathcal{C}_{\eta,0}))'_b$ is analogous, the only difference being that instead of defining $\varphi_{\Lambda'}$ as in \eqref{fun-def-for-the-est1}, one defines it by
$$
\varphi_{\Lambda'}:=\sum_{\lambda\in\Lambda'} \eta(y_{\lambda})^{-1}\psi_{\lambda}\chi_{\lambda}\in\SSS^*_{\dagger}(\RR^n),\quad \mbox{for finite}\,\, \Lambda'\subseteq\Lambda.
$$
\end{proof}

\subsection{Complex interpolation of amalgam spaces}
In this subsection, we study the complex interpolation of the amalgam spaces we defined above. We start by introducing the following notation. Let $\eta$ be a weight of class $\dagger$ and let $X$ be a Banach space. As standard, we denote by $\ell^p_{\eta}(\ZZ^n;X)$, $1\leq p\leq\infty$, the Banach space of all $X$-valued sequences $\{c_{\mathbf{n}}\}_{\mathbf{n}\in\ZZ^n}$ on $\ZZ^n$ such that $\|\{c_{\mathbf{n}}\}_{\mathbf{n}\in\ZZ^n}\|_{\ell^p_{\eta}(\ZZ^n;X)} :=(\sum_{\mathbf{n}\in\ZZ^n} \|c_{\mathbf{n}}\|_X^p \eta(\mathbf{n})^p)^{1/p}<\infty$ with the obvious modification when $p=\infty$. Additionally, we denote by $c_{\eta,0}(\ZZ^n;X)$ the closed subspace of $\ell^{\infty}_{\eta}(\ZZ^n;X)$ consisting of all $\{c_{\mathbf{n}}\}_{\mathbf{n}\in\ZZ^n}\in\ell^{\infty}_{\eta}(\ZZ^n;X)$ which satisfy $\lim_{|\mathbf{n}|\rightarrow\infty} \|c_{\mathbf{n}}\|_X\eta(\mathbf{n})=0$.

\begin{lemma}\label{lemma-retra-foramlps}
Let $\psi_1,\psi_2\in\SSS^*_{\dagger}(\RR^n)$ be such that $\{T_{\mathbf{n}}(\psi_1\psi_2)\}_{\mathbf{n}\in\ZZ^n}$ is a UCPU of class $*-\dagger$ with points $\{\mathbf{n}\}_{\mathbf{n}\in\ZZ^n}$. Let $E$ be a TMIB or a DTMIB space of class $*-\dagger$ and $\eta$ a weight of class $\dagger$. Then the mapping
\begin{equation}
\mathcal{I}:W(E,L^p_{\eta})\rightarrow \ell^p_{\eta}(\ZZ^n;E),\quad \mathcal{I}(f)=\{fT_{\mathbf{n}}\psi_1\}_{\mathbf{n}\in\ZZ^n},\qquad p\in[1,\infty],
\end{equation}
is well-defined and continuous. Moreover, for each $\{f_{\mathbf{n}}\}_{\mathbf{n}\in\ZZ^n}\in \ell^p_{\eta}(\ZZ^n;E)$, the series $\sum_{\mathbf{n}\in\ZZ^n} f_{\mathbf{n}}T_{\mathbf{n}}\psi_2$ absolutely converges in $E_{\operatorname{loc}}$ to an element of $W(E,L^p_{\eta})$ and the mapping
\begin{equation}\label{map-p-for-interpolat}
\mathcal{P}:\ell^p_{\eta}(\ZZ^n;E)\rightarrow W(E,L^p_{\eta}),\quad \mathcal{P}(\{f_{\mathbf{n}}\}_{\mathbf{n}\in\ZZ^n})=\sum_{\mathbf{n}\in\ZZ^n} f_{\mathbf{n}}T_{\mathbf{n}}\psi_2,\qquad p\in[1,\infty],
\end{equation}
is continuous. Furthermore, $\mathcal{P}\mathcal{I}=\operatorname{Id}_{W(E,L^p_{\eta})}$.\\
\indent All of the above holds true if $W(E,L^p_{\eta})$ and $\ell^p_{\eta}(\ZZ^n;E)$ are replaced by $W(E,\mathcal{C}_{\eta,0})$ and $c_{\eta,0}(\ZZ^n;E)$ respectively.
\end{lemma}

\begin{remark}
Such $\psi_1,\psi_2\in\SSS^*_{\dagger}(\RR^n)$ always exist in view of Remark \ref{rem-bupu-of-class-a-dss} and Proposition \ref{facto-s-mul}.
\end{remark}

\begin{proof} Set $\psi:=\psi_1\psi_2\in\SSS^*_{\dagger}(\RR^n)\backslash\{0\}$. The fact that $\mathcal{I}$ is well-defined and continuous can be shown by employing the same technique as in STEP 1 of the proof of Theorem \ref{the-g-voaml}. To verify that for each $\{f_{\mathbf{n}}\}_{\mathbf{n}\in\ZZ^n}\in\ell^p_{\eta}(\ZZ^n;E)$, the series $\sum_{\mathbf{n}\in\ZZ^n} f_{\mathbf{n}}T_{\mathbf{n}}\psi_2$ absolutely converges in $E_{\operatorname{loc}}$, let $\chi\in\SSS^*_{\dagger}(\RR^n)$. H\"older's inequality gives
\begin{align*}
\sum_{\mathbf{n}\in\ZZ^n} \|\chi f_{\mathbf{n}}T_{\mathbf{n}} \psi_2\|_E\leq \sum_{\mathbf{n}\in\ZZ^n} \|f_{\mathbf{n}}\|_E\|\chi T_{\mathbf{n}} \psi_2\|_{\mathcal{F}L^1_{\nu_E}}\leq C'\|\{f_{\mathbf{n}}\}_{\mathbf{n}\in\ZZ^n}\|_{\ell^p_{\eta}(\ZZ^n;E)},
\end{align*}
which shows that $\sum_{\mathbf{n}\in\ZZ^n} f_{\mathbf{n}}T_{\mathbf{n}}\psi_2$ absolutely converges in $E_{\operatorname{loc}}$ to some $f\in E_{\operatorname{loc}}$ (since $E_{\operatorname{loc}}$ is a complete l.c.s.). Next, we show that $f:=\sum_{\mathbf{n}\in\ZZ^n} f_{\mathbf{n}}T_{\mathbf{n}}\psi_2\in W(E,L^p_{\eta})$ when $\{f_{\mathbf{n}}\}_{\mathbf{n}\in\ZZ^n}\in\ell^p_{\eta}(\ZZ^n;E)$. In the Beurling case, choose $C_1,\tau>0$ such that \eqref{est-forweight-s} holds for $\eta$. For $h:=H\tau$, \eqref{chose-ofh-forconvsss} holds true. In view of Lemma \ref{est-partial-growt-fourbeuralg}, there is $C>0$ such that
\begin{equation}\label{est-for-tog-par-interpolat}
\|(T_{\mathbf{n}}\psi_2)(T_{\mathbf{m}}\psi)\|_{\mathcal{F}L^1_{\nu_E}}\leq Ce^{-A(h|\mathbf{n}-\mathbf{m}|)},\quad \mathbf{n},\mathbf{m}\in\ZZ^n.
\end{equation}
In the Roumieu case, Lemma \ref{est-partial-growt-fourbeuralg} implies that there are $C,h>0$ such that \eqref{est-for-tog-par-interpolat} holds true and we pick $C_1>0$ such that \eqref{est-forweight-s} holds true with $\tau:=h/H$; consequently, \eqref{chose-ofh-forconvsss} is also valid. The Minkowski inequality gives
\begin{align*}
\left(\sum_{\mathbf{m}\in\ZZ^n}\|fT_{\mathbf{m}}\psi\|_E^p\eta(\mathbf{m})^p\right)^{1/p} &\leq C_1\left(\sum_{\mathbf{m}\in\ZZ^n}\left(\sum_{\mathbf{n}\in\ZZ^n} \|f_{\mathbf{n}}(T_{\mathbf{n}}\psi_2)(T_{\mathbf{m}}\psi)\|_E\eta(\mathbf{n}) e^{A(\tau|\mathbf{m}-\mathbf{n}|)}\right)^p\right)^{1/p}\\
&\leq c_0CC_1\left(\sum_{\mathbf{m}\in\ZZ^n}\left(\sum_{\mathbf{n}\in\ZZ^n} \|f_{\mathbf{n}}\|_E\eta(\mathbf{n}) e^{-A(h|\mathbf{m}-\mathbf{n}|/H)}\right)^p\right)^{1/p}\\
&= c_0CC_1\left(\sum_{\mathbf{m}\in\ZZ^n}\left(\sum_{\mathbf{n}\in\ZZ^n} \|f_{\mathbf{m}-\mathbf{n}}\|_E\eta(\mathbf{m}-\mathbf{n}) e^{-A(h|\mathbf{n}|/H)}\right)^p\right)^{1/p}\\
&\leq c_0CC_1\left(\sum_{\mathbf{n}\in\ZZ^n}e^{-A(h|\mathbf{n}|/H)}\right) \|\{f_{\mathbf{n}}\}_{\mathbf{n}\in\ZZ^n}\|_{\ell^p_{\eta}(\ZZ^n;E)},
\end{align*}
with obvious modifications when $p=\infty$. This shows that $f\in W(E,L^p_{\eta})$ and that the mapping \eqref{map-p-for-interpolat} is continuous. The identity $\mathcal{P}\mathcal{I}=\operatorname{Id}_{W(E,L^p_{\eta})}$ is straightforward to verify. The proof for $W(E,\mathcal{C}_{\eta,0})$ and $c_{\eta,0}(\ZZ^n;E)$ is similar and we omit it.
\end{proof}

Before we state and prove the result on complex interpolation, we point out the following facts. If $E_0$ and $E_1$ are TMIB spaces of class $*-\dagger$ then $E_0\cap E_1$ and $E_0+E_1$ are also TMIB spaces of class $*-\dagger$; the only non-trivial part is the density of $\SSS^*_{\dagger}(\RR^n)$ in $E_0\cap E_1$ and $E_0+E_1$ which follows from \cite[Corollary 3.3]{DPPV}. Consequently, \cite[Theorem 2.7.1, p. 32]{bergh-l} implies that $E'_0\cap E'_1$ and $E'_0+ E'_1$ are DTMIB spaces of class $*-\dagger$.

\begin{theorem}
Let $\eta_0$ and $\eta_1$ be two weights of class $\dagger$ and, for $\theta\in(0,1)$, set $\eta_{\theta}(x):=\eta_0(x)^{1-\theta}\eta_1(x)^{\theta}$, $x\in\RR^n$.
\begin{itemize}
\item[$(i)$] If $E_0$ and $E_1$ are TMIB spaces of class $*-\dagger$, then $[E_0,E_1]_{[\theta]}$, $0<\theta<1$, is a TMIB space of class $*-\dagger$. For $1\leq p_0,p_1<\infty$, it holds that
    \begin{equation}\label{interpol-p0-p1}
    [W(E_0,L^{p_0}_{\eta_0}),W(E_1,L^{p_1}_{\eta_1})]_{[\theta]}=W([E_0,E_1]_{[\theta]}, L^{p_{\theta}}_{\eta_{\theta}}),\quad 0<\theta<1,
    \end{equation}
    with equivalent norms, where $p_{\theta}^{-1}=(1-\theta)p_0^{-1}+\theta p_1^{-1}$. Furthermore,
    \begin{align}
    &[W(E_0,\mathcal{C}_{\eta_0,0}),W(E_1,L^{p_1}_{\eta_1})]_{[\theta]}= W([E_0,E_1]_{[\theta]}, L^{p_1/\theta}_{\eta_{\theta}}),\quad 0<\theta<1,\label{interpol-coinf-p1}\\
    &[W(E_0,L^{p_0}_{\eta_0}),W(E_1,\mathcal{C}_{\eta_1,0})]_{[\theta]} =W([E_0,E_1]_{[\theta]}, L^{p_0/(1-\theta)}_{\eta_{\theta}}),\quad 0<\theta<1,\label{interpol-p0-coinf}\\
    &[W(E_0,\mathcal{C}_{\eta_0,0}),W(E_1,\mathcal{C}_{\eta_1,0})]_{[\theta]}= W([E_0,E_1]_{[\theta]}, \mathcal{C}_{\eta_{\theta},0}),\quad 0<\theta<1.\label{interpol-coinf-coinf}
    \end{align}
\item[$(ii)$] If $E_0$ and $E_1$ are DTMIB spaces of class $*-\dagger$ such that at least one of them is reflexive (or, equivalently, its predual is a reflexive TMIB space of class $*-\dagger$), then $[E_0,E_1]_{[\theta]}$, $0<\theta<1$, is a DTMIB space of class $*-\dagger$. Furthermore, for $1\leq p_0,p_1<\infty$, \eqref{interpol-p0-p1}, \eqref{interpol-coinf-p1}, \eqref{interpol-p0-coinf} and \eqref{interpol-coinf-coinf} hold true.
\end{itemize}
\end{theorem}

\begin{proof} We start with the following general observation. For $1\leq p_0,p_1<\infty$ and $0<\theta<1$, let $p_{\theta}$ and $\eta_{\theta}$ be as in the theorem. By employing similar technique as in the proof of \cite[Theorem 1.18.1, p. 121]{triebel}, one can show that for any (abstract) interpolation couple $\{X_0,X_1\}$, it holds that
\begin{align}
&[\ell^{p_0}_{\eta_0}(\ZZ^n;X_0),\ell^{p_1}_{\eta_1}(\ZZ^n;X_1)]_{[\theta]}= \ell^{p_{\theta}}_{\eta_{\theta}}(\ZZ^n; [X_0,X_1]_{[\theta]}),\label{int-for-lpspa-wei-withno-wgh}\\
&[c_{\eta_0,0}(\ZZ^n;X_0),\ell^{p_1}_{\eta_1}(\ZZ^n;X_1)]_{[\theta]}= \ell^{p_1/\theta}_{\eta_{\theta}}(\ZZ^n;[X_0,X_1]_{[\theta]}),\\
&[\ell^{p_0}_{\eta_0}(\ZZ^n;X_0),c_{\eta_1,0}(\ZZ^n;X_1)]_{[\theta]}= \ell^{p_0/(1-\theta)}_{\eta_{\theta}}(\ZZ^n;[X_0,X_1]_{[\theta]}),\\ &[c_{\eta_0,0}(\ZZ^n;X_0),c_{\eta_1,0}(\ZZ^n;X_1)]_{[\theta]}= c_{\eta_{\theta},0}(\ZZ^n;[X_0,X_1]_{[\theta]}).
\end{align}
We first address $(i)$. The only non-trivial part in showing that $[E_0,E_1]_{[\theta]}$ is a TMIB space of class $*-\dagger$ is the density of $\SSS^*_{\dagger}(\RR^n)$; this fact follows from \cite[Theorem 1.9.3 (c), p. 59]{triebel}. Set $\eta(x):=(\eta_0(x)+\eta_0(x)^{-1}+\eta_1(x)+\eta_1(x)^{-1})^{-1}$, $x\in\RR^n$. Clearly, $\eta$ is a weight of class $\dagger$ and $\eta(x)\leq \eta_j(x)$, $x\in\RR^n$, $j=0,1$. Let $\mathcal{I}:W(E_0+E_1,L^{\infty}_{\eta})\rightarrow \ell^{\infty}_{\eta}(\ZZ^n;E_0+E_1)$ and $\mathcal{P}:\ell^{\infty}_{\eta}(\ZZ^n;E_0+E_1)\rightarrow W(E_0+E_1,L^{\infty}_{\eta})$ be the continuous mappings defined in Lemma \ref{lemma-retra-foramlps}. Notice that, for any $p\in[1,\infty]$ and $j=0,1$, $W(E_j,L^p_{\eta_j})\subseteq W(E_0+E_1,L^{\infty}_{\eta})$ and $\ell^p_{\eta_j}(\ZZ^n;E_j)\subseteq \ell^{\infty}_{\eta}(\ZZ^n;E_0+E_1)$ continuously. To prove \eqref{interpol-p0-p1}, we point out that the above together with Lemma \ref{lemma-retra-foramlps} and the closed graph theorem imply that $\mathcal{I}$ and $\mathcal{P}$ restrict to continuous operators with the following domains and codomains:
\begin{align*}
&\mathcal{I}:W(E_0,L^{p_0}_{\eta_0})+W(E_1,L^{p_1}_{\eta_1})\rightarrow \ell^{p_0}_{\eta_0}(\ZZ^n;E_0)+\ell^{p_1}_{\eta_1}(\ZZ^n;E_1),\\
&\mathcal{P}:\ell^{p_0}_{\eta_0}(\ZZ^n;E_0)+\ell^{p_1}_{\eta_1}(\ZZ^n;E_1)\rightarrow W(E_0,L^{p_0}_{\eta_0})+W(E_1,L^{p_1}_{\eta_1}),\\
&\mathcal{I}:W(E_j,L^{p_j}_{\eta_j})\rightarrow \ell^{p_j}_{\eta_j}(\ZZ^n;E_j),\quad \mathcal{P}:\ell^{p_j}_{\eta_j}(\ZZ^n;E_j)\rightarrow W(E_j,L^{p_j}_{\eta_j}),\qquad j=0,1.
\end{align*}
In each of these cases, $\mathcal{P}\mathcal{I}=\operatorname{Id}$. Now, \cite[Theorem 1.2.4, p. 22]{triebel} yields that $\mathcal{I}$ is a topological isomorphism from $[W(E_0,L^{p_0}_{\eta_0}),W(E_1,L^{p_1}_{\eta_1})]_{[\theta]}$ onto $\mathcal{I}\mathcal{P}([\ell^{p_0}_{\eta_0}(\ZZ^n;E_0), \ell^{p_1}_{\eta_1}(\ZZ^n;E_1)]_{[\theta]})$. In view of \eqref{int-for-lpspa-wei-withno-wgh}, the latter is $\mathcal{I}\mathcal{P}(\ell^{p_{\theta}}_{\eta_{\theta}}(\ZZ^n;[E_0,E_1]_{[\theta]}))$. Because of Lemma \ref{lemma-retra-foramlps}, this is exactly $\mathcal{I}(W([E_0,E_1]_{[\theta]}, L^{p_{\theta}}_{\eta_{\theta}}))$ and the proof of \eqref{interpol-p0-p1} is complete. The rest of the claims in $(i)$ can be shown in an analogous way.\\
\indent We now address $(ii)$. The fact that $[E_0,E_1]_{[\theta]}$ is a DTMIB follows from $(i)$ and \cite[Corollary 4.5.2, p. 98]{bergh-l}. The reset of the claims in $(ii)$ can be proved in an analogous way as in the proof of $(i)$.
\end{proof}

\subsection{Independence from the sequences $\{M_p\}_{p\in\NN}$ and $\{A_p\}_{p\in\NN}$}

As we pointed out in Remark \ref{rem-for-dif-tmibspacdtspa}, a (D)TMIB space $E$ of class $*-\dagger$ can also be a (D)TMIB space for other pair of sequences $\{M_p\}_{p\in\NN}$ and $\{A_p\}_{p\in\NN}$; for example, it is always a (D)TMIB space of class $(p!)-(p!)$. We now show that in this case, the resulting amalgam spaces are always the same. For the proof in the Roumieu case we will need the following alternative description of the topology of $\SSS'^{\{M_p\}}_{\{A_p\}}(\RR^n)$ (see \cite[Section 2.5]{PilipovicK}). Let $\mathfrak{R}$ be the set of all non-decreasing positive sequences on $\ZZ_+$ that tend to $+\infty$. For each $(r_p)\in\mathfrak{R}$, denote by $A_{(r_p)}(\cdot)$ the associated function to the sequence $A_p\prod_{j=1}^pr_j$, $p\in\NN$.\footnote{Here and throughout the rest of the article we employ the principle of vacuous (empty) products, i.e. $\prod_{j=1}^0 r_j=1$.} Then $\SSS^{\{M_p\}}_{\{A_p\}}(\RR^n)$ consists of all $\varphi\in\mathcal{C}^{\infty}(\RR^n)$ such that
\begin{equation}
\sigma_{(r_p)}(\varphi):=\sup_{\alpha\in\NN^n} \frac{\|e^{A_{(r_p)}(|\cdot|)}\partial^{\alpha}\varphi\|_{L^{\infty}(\RR^n)}} {M_{\alpha}\prod_{j=1}^{|\alpha|}r_j}<\infty,\quad (r_p)\in\mathfrak{R},
\end{equation}
and the family of seminorms $\sigma_{(r_p)}$, $(r_p)\in\mathfrak{R}$, generates the topology of $\SSS^{\{M_p\}}_{\{A_p\}}(\RR^n)$.

\begin{proposition}\label{depe-donton-seamalgams}
Let $E$ be a TMIB or a DTMIB space of class $*-\dagger$ and let $\eta$ be a weight of class $\dagger$. Then the amalgam spaces $W(E,L^p_{\eta})$, $1\leq p\leq \infty$, and $W(E,\mathcal{C}_{\eta,0})$ are the same when defined via $\SSS^*_{\dagger}(\RR^n)-\SSS'^*_{\dagger}(\RR^n)$ or via $\SSS^{(p!)}_{(p!)}(\RR^n)- \SSS'^{(p!)}_{(p!)}(\RR^n)$.
\end{proposition}

\begin{proof} Without loss of generality, we can assume that $\eta$ is continuous. In view of Proposition \ref{tmib}, it suffices to show that $\widetilde{W}(E,L^{\infty}_{\eta})\subseteq E_{\operatorname{loc}}$, where $\widetilde{W}(E,L^{\infty}_{\eta})$ is the amalgam space defined via $\SSS^{(p!)}_{(p!)}(\RR^n)- \SSS'^{(p!)}_{(p!)}(\RR^n)$ and $E_{\operatorname{loc}}$ is defined via $\SSS^*_{\dagger}(\RR^n)- \SSS'^*_{\dagger}(\RR^n)$. We prove this only in the Roumieu case as the Beurling case is similar. Let $\{T_{\mathbf{n}}\psi\}_{\mathbf{n}\in\ZZ^n}$ be a UCPU of class $(p!)-(p!)$ as constructed in Remark \ref{rem-bupu-of-class-a-dss}. In view of Proposition \ref{facto-s-mul}, there are $\psi_1,\psi_2\in\SSS^{(p!)}_{(p!)}(\RR^n)$ such that $\psi=\psi_1\psi_2$. First we show that $\widetilde{W}(E,L^{\infty}_{\eta})\subseteq \SSS'^{\{M_p\}}_{\{A_p\}}(\RR^n)$. If $E$ is a TMIB, set $F:=E'$, and, when $E$ is a DTMIB, let $F$ be the TMIB so that $E=F'$. Since $\SSS^{\{M_p\}}_{\{A_p\}}(\RR^n)$ is continuously included into $F$, there are $C'>0$ and $(r'_p)\in\mathfrak{R}$ such that $\|\chi\|_F\leq C'\sigma_{(r'_p)}(\chi)$, $\chi\in\SSS^{\{M_p\}}_{\{A_p\}}(\RR^n)$. In view of \cite[Lemma 2.1]{PP3}, there are $C''>0$ and $(r''_p)\in\mathfrak{R}$ such that $1/\eta(x)\leq C''e^{A_{(r''_p)}(|x|)}$, $x\in\RR^n$. Employing \cite[Lemma 2.3]{bojan}, one can find $(r_p)\in\mathfrak{R}$ such that $r_p\leq \min\{r'_p,r''_p\}$, $p\in\ZZ_+$, and $\prod_{j=1}^{m+k}r_j\leq 2^{m+k}(\prod_{j=1}^m r_j) (\prod_{j=1}^k r_j)$, $m,k\in\ZZ_+$. Hence, the sequence $A_p\prod_{j=1}^pr_j$, $p\in\NN$, satisfies $(M.2)$ with a constant $2H$ instead of $H$. Set $\tilde{r}_p:=r_p/8H^2$, $p\in\ZZ_+$. Let $f\in \widetilde{W}(E,L^{\infty}_{\eta})$ and $\chi\in\SSS^{(p!)}_{(p!)}(\RR^n)$. Then (notice that $\sigma_{(r'_p)}\leq \sigma_{(r_p)}$)
$$
|\langle f,\chi\rangle|\leq \sum_{\mathbf{n}\in\ZZ^n} |\langle fT_{\mathbf{n}}\psi_1, \chi T_{\mathbf{n}}\psi_2\rangle|\leq \sum_{\mathbf{n}\in\ZZ^n}\|fT_{\mathbf{n}}\psi_1\|_E \|\chi T_{\mathbf{n}}\psi_2\|_F\leq C_1\sum_{\mathbf{n}\in\ZZ^n} \frac{\sigma_{(r_p)}(\chi T_{\mathbf{n}}\psi_2)}{\eta(\mathbf{n})}.
$$
We estimate as follows
\begin{align*}
&\frac{|\partial^{\alpha}(\chi(x) T_{\mathbf{n}}\psi_2(x))|e^{A_{(r_p)}(|x|)}} {\eta(\mathbf{n})M_{\alpha}\prod_{j=1}^{|\alpha|}r_j}\\
&\leq C''\sum_{\beta\leq \alpha} {\alpha\choose \beta} \frac{|\partial^{\alpha-\beta}\chi(x)||\partial^{\beta} \psi_2(x-\mathbf{n})|e^{A_{(r_p)}(|x|)} e^{A_{(r_p)}(|\mathbf{n}|)}} {M_{\alpha-\beta} (\prod_{j=1}^{|\alpha-\beta|}r_j) M_{\beta}(\prod_{j=1}^{|\beta|} r_j)}\\
&\leq C'' 8^{-|\alpha|}\sigma_{(\tilde{r}_p)}(\psi_2) \sigma_{(\tilde{r}_p)}(\chi) e^{-A_{(r_p)}(8H^2|x|)} e^{-A_{(r_p)}(8H^2|x-\mathbf{n}|)} e^{A_{(r_p)}(|x|)} e^{A_{(r_p)}(|\mathbf{n}|)}\sum_{\beta\leq \alpha} {\alpha\choose \beta}\\
&\leq C_2 \sigma_{(\tilde{r}_p)}(\psi_2) \sigma_{(\tilde{r}_p)}(\chi) e^{-A_{(r_p)}(4H|x|)} e^{-A_{(r_p)}(8H^2|x-\mathbf{n}|)} e^{A_{(r_p)}(|\mathbf{n}|)}\\
&\leq C_3 \sigma_{(\tilde{r}_p)}(\psi_2) \sigma_{(\tilde{r}_p)}(\chi) e^{-A_{(r_p)}(2H|\mathbf{n}|)} e^{A_{(r_p)}(|\mathbf{n}|)} \leq  C_4 \sigma_{(\tilde{r}_p)}(\psi_2) \sigma_{(\tilde{r}_p)}(\chi) e^{-A_{(r_p)}(|\mathbf{n}|)}.
\end{align*}
Hence $|\langle f,\chi\rangle|\leq C_5 \sigma_{(\tilde{r}_p)}(\chi)$ and, as $\SSS^{(p!)}_{(p!)}(\RR^n)$ is dense in $\SSS^{\{M_p\}}_{\{A_p\}}(\RR^n)$, we infer $f\in\SSS'^{\{M_p\}}_{\{A_p\}}(\RR^n)$.\\
\indent Now, to prove $\widetilde{W}(E,L^{\infty}_{\eta})\subseteq E_{\operatorname{loc}}$, let $f\in\widetilde{W}(E,L^{\infty}_{\eta})$ and $\chi\in\SSS^{\{M_p\}}_{\{A_p\}}(\RR^n)$. In view of the above, $f\chi= \sum_{\mathbf{n}\in\ZZ^n} f\chi T_{\mathbf{n}}\psi$ and the series is absolutely convergent in $\SSS'^{\{M_p\}}_{\{A_p\}}(\RR^n)$. As $f\chi T_{\mathbf{n}}\psi= (f T_{\mathbf{n}}\psi_1)(\chi T_{\mathbf{n}}\psi_2)\in E$, we infer
$$
\sum_{\mathbf{n}\in\ZZ^n} \|f\chi T_{\mathbf{n}}\psi\|_E\leq C_6\sum_{\mathbf{n}\in\ZZ^n} \|\chi T_{\mathbf{n}}\psi_2\|_{\mathcal{F}L^1_{\nu_E}}/\eta(\mathbf{n})<\infty,
$$
where the convergence of the very last series easily follows from Lemma \ref{lemma-for-inc-dl1-f}. Hence $f \chi\in E$ and the proof is complete.
\end{proof}

\begin{remark}\label{rem-for-amalg-spac-not-deponseq}
As a direct consequence, we deduce that the amalgam spaces do not depend on the sequences $\{M_p\}_{p\in\NN}$ and $\{A_p\}_{p\in\NN}$ in the following sense. Let $\{\widetilde{M}_p\}_{p\in\NN}$ and $\{\widetilde{A}_p\}_{p\in\NN}$ be two sequences that satisfy the same conditions as $\{M_p\}_{p\in\NN}$ and $\{A_p\}_{p\in\NN}$ respectively. We employ $\SSS'^{\widetilde{*}}_{\widetilde{\dagger}}(\RR^n)$ and $\SSS^{\widetilde{*}}_{\widetilde{\dagger}}(\RR^n)$ as a common notation for the Beurling and Roumieu variant of the space of tempered ultradistributions and the corresponding test space associated to the sequences $\{\widetilde{M}_p\}_{p\in\NN}$ and $\{\widetilde{A}_p\}_{p\in\NN}$. Let $E$ be a TMIB or a DTMIB space of class $*-\dagger$ and of class $\widetilde{*}-\widetilde{\dagger}$ simultaneously and let $\eta$ be a weight of class $\dagger$ and of class $\widetilde{\dagger}$ simultaneously. Then Proposition \ref{depe-donton-seamalgams} yields that the amalgam spaces $W(E,L^p_{\eta})$, $1\leq p\leq \infty$, and $W(E,\mathcal{C}_{\eta,0})$ are the same when defined via $\SSS^*_{\dagger}(\RR^n)-\SSS'^*_{\dagger}(\RR^n)$ or via $\SSS^{\widetilde{*}}_{\widetilde{\dagger}}(\RR^n)- \SSS'^{\widetilde{*}}_{\widetilde{\dagger}}(\RR^n)$.
\end{remark}

\begin{remark}
Let $E$ be a (D)TMIB space of distributions and let $\eta$ be a polynomially bounded weight. Then $E$ is also a (D)TMIB space of ultradistributions of class $*-\dagger$ and, similarly as above, one can show that the amalgam spaces $W(E,L^p_{\eta})$, $1\leq p\leq \infty$, and $W(E,\mathcal{C}_{\eta,0})$ when defined via $\SSS^*_{\dagger}(\RR^n)-\SSS'^*_{\dagger}(\RR^n)$ are the same with the corresponding (standard) amalgam space of distributions (cf. \cite[Section 3]{f-p-p}).
\end{remark}

Finally, we point out that when $A_p=M_p$, $p\in\NN$, most of the amalgam spaces we considered above can be viewed as generalised modulation spaces as defined in \cite[Section 4.2]{DPPV}. To make this precise, let $\eta$ be a weight of class $*$ on $\RR^n$. If $E$ is a TMIB space of class $*-*$ on $\RR^n$, then the weighted Bochner-Lebesgue space $L^p_{\eta}(\RR^n_x;\mathcal{F}E_{\xi})$, $1\leq p<\infty$, is a TMIB space of class $*-*$ on $\RR^{2n}$; see \cite[Remark 3.9]{DPPV}. Hence, one can consider the generalised modulation space $\mathcal{M}^{L^p_{\eta}(\RR^n_x;\mathcal{F}E_{\xi})}$, $1\leq p<\infty$, as defined in \cite[Section 4.2]{DPPV}. Then
\begin{equation}\label{equ-mod-amal-wihtnew}
\mathcal{M}^{L^p_{\eta}(\RR^n_x;\mathcal{F}E_{\xi})}=W(E,L^p_{\eta}),\quad 1\leq p<\infty.
\end{equation}
Employing the identity $V_{\varphi}f(x,\xi)=\mathcal{F}(fT_x\overline{\varphi})(\xi)$, $x,\xi\in\RR^n$, $f\in\SSS'^*_*(\RR^n)$, where $V_{\varphi}$ is the short-time Fourier transform with window $\varphi\in\SSS^*_*(\RR^n)\backslash\{0\}$, it is straightforward to verify that both of the spaces in \eqref{equ-mod-amal-wihtnew} induce the same topology on $\SSS^*_*(\RR^n)$. As the latter is dense in both of the spaces in \eqref{equ-mod-amal-wihtnew} (because of Proposition \ref{tmib} and \cite[Theorem 4.8 $(i)$]{DPPV}), the identity follows. If, in addition, $E'$ satisfies the Radon-Nikodym property (in particular, when $E$ is reflexive), the strong dual of $L^q_{1/\eta}(\RR^n_x;\mathcal{F}E_{\xi})$, $1\leq q<\infty$, is $L^p_{\eta}(\RR^n_x;(\mathcal{F}E')\check{}_{\xi})$, $p^{-1}+q^{-1}=1$, where $(\mathcal{F}E')\check{}$ is the DTMIB space $\{e\in\SSS'^*_*(\RR^n)\,|\, \check{e}\in \mathcal{F}E'\}$ with norm $\|e\|_{(\mathcal{F}E')\check{}}:=\|\check{e}\|_{\mathcal{F}E'}$, $e\in(\mathcal{F}E')\check{}$. Consequently, $L^p_{\eta}(\RR^n_x;(\mathcal{F}E')\check{}_{\xi})$, $1<p\leq \infty$, is a DTMIB space of class $*-*$ on $\RR^{2n}$ and \eqref{equ-mod-amal-wihtnew} together with Theorem \ref{duali} and \cite[Theorem 4.8 $(iii)$]{DPPV} give
\begin{equation}
\mathcal{M}^{L^p_{\eta}(\RR^n_x;\mathcal{F}E'_{\xi})}=W(E',L^p_{\eta}),\quad 1< p\leq \infty.
\end{equation}
\\
\\
\noindent \textbf{Acknowledgements.} We thank the anonymous reviewer whose valuable comments and suggestion have greatly improved the article.

\end{document}